\documentclass[10pt]{article}
\usepackage{amsmath}
\usepackage{amsfonts}
\usepackage{amssymb}
\usepackage{graphicx}
\usepackage{mathrsfs}
\usepackage{xcolor}
\usepackage{bbm}
\usepackage{verbatim}
\usepackage[T1]{fontenc}
\usepackage{dsfont}
\usepackage{mathrsfs}
\usepackage[body={15.5cm,21cm}, top=3cm]{geometry}
\usepackage{paralist}
\usepackage{indentfirst}
\allowdisplaybreaks[4]
\usepackage{hyperref}
\providecommand{\U}[1]{\protect \rule{.1in}{.1in}}
\numberwithin{equation}{section}
\newtheorem{Theorem}{Theorem}[section]

\newtheorem{lemma}[Theorem]{Lemma}

\newtheorem{problem}[Theorem]{Problem}

\newtheorem{remark}[Theorem]{Remark}

\newenvironment{proof}[1][Proof]{\noindent \textbf{#1.} }{\  \rule{0.5em}{0.5em}}
\allowdisplaybreaks[2]
\begin{document}
\title{Linear-quadratic extended mean field games with common noises}
\author{Tianjiao Hua \thanks{School of Mathematical Sciences, Shanghai Jiao Tong University, China (htj960127@sjtu.edu.cn)}
\and
Peng Luo \thanks{School of Mathematical Sciences, Shanghai Jiao Tong University, China (peng.luo@sjtu.edu.cn). Research supported
by the National Natural Science Foundation of China (No. 12101400).}}
\maketitle
\begin{abstract}
In this paper, we consider a class of linear quadratic extended mean field games (MFGs) with common noises where the state coefficients and the cost functional vary with the mean field term in a nonlinear way. Based on stochastic maximum principle, solving the mean field game is transformed into solving a conditional mean field forward-backward stochastic differential equation (FBSDE). We first establish solvability for a type of (more general) conditional mean field FBSDEs under monotonicity conditions. We further provide some regularity results which lead to classical solutions for the associated master equations. In particular, the linear quadratic extended mean field game is solved and classical solution for (extended mean field game) master equation is obtained.
\end{abstract}


\textbf{MSC-classification}: 60H30, 93H20, 49N70.
\section{Introduction}
\indent   Mean field games (MFGs for short) were independently proposed by Lasry and Lions in \cite{lasry2007mean} from a partial differential equation perspective and Huang, Caines, and Malham\'{e} \cite{huang2007large} from an engineering perspective. Since it provides an effective method to enable each agent to take advantage of mean field interaction to avoid the difficulty of dimensionality, MFGs theory has grown tremendously. Carmona and Delarue approached the MFG problems from a probabilistic point of view (see e.g. \cite{carmona2013probabilistic,carmona2018probabilistic,10.1214/14-AOP946}). When the mean field interaction is not only via the distribution of the states but also controls, this MFG system is called extended mean field game, also known as mean field game of controls. Extended mean field games have wide applications in practical problems, such as trade
crowding, price impact model, see, e.g. \cite{carmona2018probabilistic,carmona2021probabilistic,alasseur2020extended,cardaliaguet2018mean}. 

In this paper, we aim to 
 establish a unique equilibrium for a type of linear quadratic extended mean field games with common noises, where
the mean field interaction is only through the (conditional) expectations of the state and control and is allowed to vary in a nonlinear way in the state coefficients and cost functional. Using stochastic maximum principle and consistency condition, we show that the analysis of the extended mean field games with common noises reduces to solving conditional mean field FBSDE of the form
\begin{equation} \label{fbsde}
    \left\{
    \begin{aligned}
dX_{t} &= \xi+\int_{0}^{t}\Big[b_{1}(s)X_{s}+b_{2}(s)Y_{s}+b_{0}\left(s,E[X_{s}|\mathcal{F}^{W^{0}}_{s}],E[X_{s}|\mathcal{F}^{W^{0}}_{s}]\right)\Big]ds+\int_{0}^{t}\sigma(s)dW_{s}\\
&\quad +\int_{0}^{t}\Big[\sigma_{1}(s)X_{s}+ Y_{s}^{\mathsf{T}}\sigma_{2}(s)+\sigma_{0}\left(s,E[X_{s}|\mathcal{F}^{W^{0}}_{s}],E[Y_{s}|\mathcal{F}^{W^{0}}_{s}]\right) \Big]dW_{s}^{0},\\
dY_{t} &=h_{1}X_{T}+h_{2}\left(\mathbb{E}[X_{T}|\mathcal{F}_{T}^{W^{0}}]\right)+\int_{t}^{T} \Big[f_{1}(s)X_{s}+f_{2}(s)Y_{s}+f_{0}\left(s,E[X_{s}|\mathcal{F}^{W^{0}}_{s}],E[X_{s}|\mathcal{F}^{W^{0}}_{s}]\right)\Big]ds\\
&\quad -\int_{t}^{T}Z_{s}dW_{s}-\int_{t}^{T}Z^{0}_{s}dW_{s}^{0}.
\end{aligned}\right.
\end{equation}
Here, $W$ and $W^{0}$ are two independent multidimensional Brownian motions and $(\mathcal{F}^{W^{0}}_{t})_{t\geq 0}$ is the filtration generated by $W^{0}$. Such conditional mean field FBSDE \eqref{fbsde} plays an important role in mean field games and mean field control problems when the mean field interaction is only through the (conditional) expectations of the state and control.

The presence of the common noise clearly brings extra complexity to the problem since the mean field interaction becomes stochastic rather than deterministic. Therefore, the theoretical literature on MFGs with common noise remained relatively limited until recent years, not to mention extended MFGs with common noises. To the best of our knowledge, there are mainly three methods to deal with MFGs with common noises. The first method  is continuation method. Ahuja et al. \cite{ahuja2019forward} used lifting technique to convert conditional mean field  FBSDEs to functional FBSDEs and give its well-posedness by continuation method. The second method is to use the contraction of control in specific optimal control problem situation, such as \cite{ahuja2016wellposedness} and \cite{huang2022mean}, where they apply the Banach fixed point theorem to show existence over small duration and attempt to extend the solution over arbitrary time. \cite{ahuja2016wellposedness} considered a simple structure which is extended to a more general structure with similar method in \cite{huang2022mean}. The third method is presenting a theory establishing the existence of strong solutions without relying on uniqueness under submodularity conditions by Tarski's fixed point theory, see, e.g., \cite{dianetti2022strong,dianetti2021submodular,dianetti2022unifying}. However, they fail to obtain the uniqueness of the solution. There are also some studies on the weak mean field equilibria when studying MFGs with common noises, initiated in \cite{carmonaweak} and extended in  \cite{djete2023mean,carmona2018probabilistic,barraso2022}.

Our main contributions in this paper are the following: first, we prove the existence and uniqueness of solution for the conditional mean field FBSDE \eqref{fbsde}, which encompasses the Hamiltonian system arising from the linear quadratic extended mean field games with common noises; second, we obtain (globally) classical solutions for associated master equations,  which incorporates the master equations for the extended mean field games with common noises. Our technique to solve the conditional mean field FBSDE \eqref{fbsde} includes two steps. In the first step, we fix the conditional expectation terms in the coefficients and obtain the global solvability for a classical FBSDE. In the second step, we introduce a map and convert finding the fixed point to solving another classical FBSDE. In both steps, we make full use of the method imposed by \cite{hua2022unified} to get the well-posedness of the corresponding FBSDEs.
It is worth mentioning that the unique solution satisfies a uniformly Lipschitz continuity property, which further contributes to guarantee the regularity of the decoupling field of conditional mean field FBSDE \eqref{fbsde}. Based on this and following the probabilistic perspective introduced by Chassagneux, Crisan and Carmona in \cite{chassagneux2014probabilistic}, we are able to establish the global solvability of the related master equations. Specifically, we show that the decoupling field of conditional mean field FBSDE is indeed the unique classical solution of the master equation based on the regularity of the decoupling field. 

  Since introduced by Lions' lectures \cite{lions2007cours} at \textit{Coll{\`e}ge de France}, master equations have played an increasingly important role in studying equilibrium of mean field games or mean field control problems. When the time duration is small enough and the data are sufficiently smooth, master equation usually admits a classical solution, see, e.g. \cite{bensoussan2019control,gangbo2015existence,cardaliaguet2022splitting}. However, the global solvability of master equation is more challenging. Monotonicity conditions play key role in the global well-posedness of the master equation as well as the uniqueness of mean field equilibria and solutions to mean field game systems. The famous one among them is Lasry-Lions monotonicity condition, which has been used to establish the global well-posedness of mean field games master equations in various setting, we refer the reader to \cite{chassagneux2014probabilistic,carmona2018probabilistic,cardaliaguet2019master,gangbo2022global,bertucci2021master}. It should be noted that the above results not only require that Lasry-Lions monotonicity is satisfied but also that the Hamiltonians is separable. Recently, for nonseparable Hamiltonians, two other criterions known as displacement monotonicity and anti-monotinicity are proposed to obtain the global well-posedness of classical solution for the master equations, see, e.g. \cite{ahuja2016wellposedness,gangbo2022global,gangbo2022mean,huang2022mean,mou2022mean}. More recently, \cite{graber2023110095} proposed two new types of monotonicity conditions based on the global-in-time optimal response to an arbitrary crowd trajectory to ensure the uniqueness of equilibrium of mean field games but they did not consider the well-posedness of master equations. In this paper, we give a new type of monotonicity conditions when the mean field interaction is only via the (conditional) expectations of states and controls. The monotonicity condition on the one hand, ensures the uniqueness of equilibrium of the extended MFGs with common noises, and on the other hand, allows us to obtain the global well-posedness of classical solutions to the related master equation with nonseparable Hamiltonian. Moreover, it is worth emphasizing that our monotonicity conditions operate on the function of the expectation, not the measure and  are in dichotomy with the Lasry-Lions monotonicity and displacement monotonicity conditions. 

Let us now compare our results to those in the existing literature and point out the novelties of
our paper. As mentioned above, \cite{huang2022mean} and \cite{ahuja2019forward} also investigated MFGs with common noises. Compared with \cite{huang2022mean}, we remove the restrictions that the
dependence of the state coefficient on the conditional distribution is sufficiently small, or 
the convexity parameter of the running cost on the control is sufficiently large. \cite{ahuja2019forward} required that the state dynamics does not depend on the conditional distributions of the state to ensure the monotonicity conditions hold, which limits the application scope of the method to practical problems. Moreover, \cite{huang2022mean} and \cite{ahuja2019forward}  only include the interaction of states and in this paper, we also consider the interaction with controls. From the perspective of global solvability of master equation, we consider a situation that does not appear in \cite{chassagneux2014probabilistic}, the master equation for extended MFGs with common noises. By using a different technique to prove the regularity of the decoupling field, we do not need the Hamiltonian to be separable and remove the boundedness assumption for the diffusion process. Compared with our previous work \cite{hua2023well}, we extend the master equations from first order to second order in the presence of common noises. The most closely related paper is \cite{li2023linear}, which also investigated a linear quadratic extended mean field game with common noise. By contrast, we consider a more general form of state dynamics and running cost functional. Following different technique, we do not need the boundedness assumption for the nonlinear functional and the non-degeneracy with respect to the common noise. For detailed comparison, please see remark \ref{comparison remark}. 

The rest of the paper is organized as follows. In section 2, we analyse a class of extended mean field games with common noises and establish the global well-posedness of the stochastic Hamiltonian system and the corresponding master equations. The details of the proof are presented in section 3, where we consider a type of (more general) conditional mean field FBSDEs and associated master equations. 

\textbf{Notations and Conventions.}
 Let $(\Omega, \mathcal{F}, \mathbb{F}, \mathbb{P})$ be a complete filtered probability space which can support two independent d-dimensional Brownian motions: $W$ and $W^0$.
For any filtration $\mathbb{G}$, we introduce the following spaces: 
$\beta \in$ $L_{\mathbb{G}}^2([0, T] ; \mathbb{R}^{n})$ if $\beta: \Omega \times[0, T] \rightarrow \mathbb{R}^{n}$ is a $\mathbb{G}$-progressively measurable process such that $\mathbb{E}\left[\int_0^T\left|\beta_t\right|^2 d t\right]$ $<\infty$; $\alpha \in S^{2}_{\mathbb{G}}([0,T];\mathbb{R}^{n})$ if $\alpha: \Omega \times[0, T] \rightarrow \mathbb{R}^{n}$ is a $\mathbb{G}$-progressively measurable process such that
$\mathbb{E}\left[ \sup\limits _{0 \leq t \leq T}|\alpha_{t}|^{2}\right] < \infty$. For any $\sigma$-field $\mathcal{G}$, we denote $\xi \in L_{\mathcal{G}}^2$ if $\xi: \Omega \rightarrow \mathbb{R}$ is a $\mathcal{G}$-measurable random variable such that $\mathbb{E}\left[|\xi|^2\right]<\infty$.

Unless otherwise stated, all equalities and inequalities between random variables and processes will be understood in the $\mathbb{P}$-a.s. and $\mathbb{P} \otimes d t$-a.e. sense, respectively. $|\cdot|$ denotes the Euclidean norm. $C^k(\mathbb{R}^{n};\mathbb{R}^k)$ denotes the space of all $\mathbb{R}^k$-valued and continuous functions $f$ on $\mathbb{R}^{n}$ with continuous derivatives up to order $k$.
$C^0([0,T] \times \mathbb{R}^{n};\mathbb{R}^k)$ denotes the space of all $\mathbb{R}^k$-valued and continuous functions $f$ on $[0,T] \times \mathbb{R}^{n}$. $C^1([0,T] \times \mathbb{R}^{n};\mathbb{R}^k)$ denotes the space of all $\mathbb{R}^k$-valued and continuous functions $f$ on $[0,T] \times \mathbb{R}^{n}$  whose partial derivatives $\frac{\partial f}{\partial t}, \frac{\partial f}{\partial x_i},  1 \leq i \leq n$, exist and are continuous. 
$C^{1,2}([0, T] \times \mathbb{R}^{n};\mathbb{R}^k)$ denotes the space of  all $\mathbb{R}^k$-valued and continuous functions $f$ on $[0, T] \times \mathbb{R}^{n}$ whose partial derivatives $\frac{\partial f}{\partial t}, \frac{\partial f}{\partial x_i}, \frac{\partial^2 f}{\partial x_i \partial x_j}, 1 \leq i, j \leq n$, exist and are continuous. For $x, y \in \mathbb{R}^{n}, x \leq y$ is understood component-wisely, i.e., $x \leq y$ if and only if $x^{i} \leq y^{i}$ for all $i=1, \ldots, n$. Throughout the paper, for any $x \in \mathbb{R}$ and any function $\phi(x)$, we will use the following convention
 \begin{equation*}
    \frac{\phi(x)-\phi(x)}{x-x}:=0.
\end{equation*}
\section{Extended mean field games with common noises}

In this section, we investigate a class of extended mean field games with common noises. First in subsection \ref{section 2.1} and subsection \ref{section 2.2}, we formulate the extended mean field games from N-player stochastic differential games of controls. Then in subsection \ref{section 2.3} and \ref{section 2.4}, the optimization problem is transformed into solving a conditional mean field FBSDE and we provide sufficient conditions on the data for the existence and uniqueness of solutions to \eqref{consistency fbsde} and the corresponding master equation \eqref{MP master equation}, which follow immediately from the well-posedness of a type of (more general) conditional mean field FBSDEs \eqref{conditional fbsde} and associated master equations \eqref{general master equation} established in Section \ref{section 3}.
\subsection{N-player stochastic differential games of controls}\label{section 2.1}
 In this subsection, we consider a class of linear-quadratic (LQ) $N$-player games of controls. For a given $T>0$, let $(\widetilde{\Omega}, \tilde{\mathcal{F}}, \widetilde{\mathbb{F}}, \widetilde{\mathbb{P}})$ be a complete filtered probability space which can support $N+1$ independent d-dimensional Brownian motions: $W^i, 1 \leq i \leq N$, and $W^0$. Here $W^i$ denotes the idiosyncratic noises for the $i-$th player and $W^0$ denotes the common noise for all the players. Let $\widetilde{\mathbb{F}}:=\left\{\widetilde{\mathcal{F}}_t\right\}_{0 \leq t \leq T}$ where $\widetilde{\mathcal{F}}_t:=\left(\vee_{i=1}^N \mathcal{F}_t^{W^i}\right) \vee \mathcal{F}_t^{W^0} \vee \tilde{\mathcal{F}}_0$ and let $\widetilde{\mathbb{P}}$ have no atom in $\tilde{\mathcal{F}}_0$.

For $1 \leq i \leq N$, let $\xi^i \in L_{\widetilde{\mathcal{F}}_0}^2$ be independent and identically distribution (i.i.d.) random variables. Denote $x:=\left(x^1, \ldots, x^N\right)$ and $\alpha:=\left(\alpha^1, \ldots, \alpha^N\right)$. The dynamic of $i$-th player's state process $x^{i} \in \mathbb{R}$ is 
\begin{equation*}
\left\{
\begin{aligned}
 d x_t^i&=\left[A_{t} x_t^i+B_{t} \alpha_t^i+f\left(t,\nu_{\boldsymbol{x}_t}^{N, i}\right)+b\left(t,\mu_{\alpha_t}^{N, i}\right)\right] d t+\sigma d W_t^i+\sigma_0 d W_t^0, \\
x_0^i&=\xi^i,   
\end{aligned}\right.
\end{equation*}
where $\alpha^i \in \widetilde{\mathcal{U}}_{a d}[0, T]=\left\{\alpha \mid \alpha \in L_{\widetilde{\mathbb{F}}}^2([0, T] ; \mathbb{R})\right\}$, $\sigma,\sigma_{0}$ are constant vectors, and
\begin{equation*}
A_{t},B_{t}:[0,T]\rightarrow \mathbb{R},\quad
    f,b:[0,T]\times \mathbb{R}\rightarrow \mathbb{R}.
    \end{equation*}
The interactions among players are via the average of all other players' states and controls
\begin{equation*}
\mu_{\boldsymbol{\alpha}}^{N, i}:=\frac{1}{N-1} \sum_{j \neq i} \alpha^j, \quad \nu_{\boldsymbol{x}}^{N, i}:=\frac{1}{N-1} \sum_{j \neq i} x^j.
\end{equation*}
The cost functional of $i$-th player is assumed to be
\begin{equation*}
\begin{aligned}
\mathcal{J}^i\left(\alpha^i, \boldsymbol{\alpha}^{-i}\right):= & \frac{1}{2} \mathbb{E}\Bigg\{\int_0^T\left[Q_{t}\left(x_t^i+l\left(t,\nu_{\boldsymbol{x}_t}^{N, i}\right)\right)^2+R_{t}\left(\alpha_t^i+h\left(t,\mu_{\boldsymbol{\alpha}_t}^{N, i}\right)\right)^2+2F_{t}x_{t}^{i}(\alpha_{t}+q\left(t,\mu_{\boldsymbol{\alpha}_t}^{N, i}\right))\right] d t\\
& \quad\quad\quad +G\left(x_T^i+g\left(\nu_{\boldsymbol{x}_T}^{N, i}\right)\right)^2\Bigg\},
\end{aligned}
\end{equation*}
where $\boldsymbol{\alpha}^{-i}=\left(\alpha^1, \ldots, \alpha^{i-1}, \alpha^{i+1}, \ldots, \alpha^N\right)$ denotes a strategy profile of other players excluding the $i$-th player, G is a constant and
\begin{equation*}
    Q_{t},R_{t},F_{t} : [0,T] \rightarrow \mathbb{R},\quad 
    l,h,q:[0,T]\times \mathbb{R}\rightarrow \mathbb{R},\quad g:\mathbb{R}\rightarrow \mathbb{R}.
\end{equation*}

 The above $N$-player game problem is to find the Nash equilibrium $\alpha^*$, that is to find a strategy profile $\alpha^*=\left(\alpha^{*, 1}, \ldots, \alpha^{*, N}\right)$ where $\alpha^{*, i}$ $\in \widetilde{\mathcal{U}}_{a d}(0, T), 1 \leq i \leq N$, such that
\begin{equation*}
\mathcal{J}^i\left(\alpha^{*, i}, \boldsymbol{\alpha}^{*,-i}\right)=\inf _{\alpha^i \in \widetilde{\mathcal{U}}_{a d}(0, T)} \mathcal{J}^i\left(\alpha^i, \boldsymbol{\alpha}^{*,-i}\right).
\end{equation*}

We should mention that when the functions $A_{t}, B_{t}, Q_{t}, R_{t}, f(t,\cdot),b(t,\cdot),l(t,\cdot),h(t,\cdot)$ are independent of $t$, \cite{li2023linear} has investigated similar N-player stochastic differential games of controls. However, they didn't take into account cross term and can only deal with non-degenerate situations. To the best of our knowledge, if the non-degenerate condition is removed, there is no result yet in the N-player games. Therefore, we will consider the limit problem, which can be solved by our approach without the non-degenerate condition. The N-play games without non-degenerate condition are left for future work.
\subsection{Formulation of the extended MFGs with common noises}\label{section 2.2}
We now formulate the extended MFGs with common noises by taking the limit of the N-player stochastic differential games as $N\rightarrow \infty$. Let $(\Omega, \mathcal{F}, \mathbb{F}, \mathbb{P})$ be a complete filtered probability space which can support two independent d-dimensional Brownian motions: $W$ and $W^0$. Here $W$ denotes the idiosyncratic noise and $W^0$ denotes the common noise. We let $\mathbb{F}:=\left\{\mathcal{F}_t\right\}_{t \in[0, T]}$, where $\mathcal{F}_t:=\mathcal{F}_t^W \vee \mathcal{F}_t^{W^0} \vee \mathcal{F}_0$, and let $\mathbb{P}$ have no atom in $\mathcal{F}_0$ so it can support any measure on $\mathbb{R}$ with a finite second moment. We denote $\mathbb{F}^0:=\{\mathcal{F}_t^{W^0}\}_{t \in[0, T]}$, $ \mathcal{U}_{a d}(0, T):=\left\{\alpha \mid \alpha \in L_{\mathbb{F}}^2([0, T] ; \mathbb{R})\right\}$. The problem of extended MFGs with common noises is defined as follows and the functions appearing are as provided in the previous subsection.
\begin{problem}\label{problem}
   Find an optimal control $\alpha^{*} \in \mathcal{U}_{a d}(0, T) $ for the stochastic control problem 
   \begin{equation}
   \left\{\begin{aligned}
        &\alpha^{*} \in \underset{\alpha \in \mathcal{U}_{a d}(0, T)}{\operatorname{argmin}} J(\alpha \mid \mu,\nu):=\frac{1}{2} \mathbb{E}\Bigg\{\int_0^T\Big[Q_{t}\left(x_t^{\xi, \alpha}+l\left(t,\nu_t\right)\right)^2+R_{t}\Big(\alpha_t+h\left(t,\mu_t\right)\Big)^2\\
   &\quad \quad \quad \quad 
 \quad  \quad 
\quad \quad \quad \quad \quad \quad \quad \quad  \quad +2F_{t}x_{t}^{\xi, \alpha}\Big(\alpha_{t}+q(t,\mu_{t})\Big)\Big] d t+G\left(x_T^{\xi, \alpha}+g\left(\nu_T\right)\right)^2\Bigg\};\\
    & x_{t}^{\xi,\alpha} = \xi +\int_{0}^{t}\left(A_{s}x_{s}^{\xi,\alpha}+B_{s}\alpha_{s}+f(s,\nu_{s})+b(s,\mu_{s})\right)ds+\sigma dW_{s}+\sigma_{0}dW^{0}_{s},\xi 
    \in L^{2}_{\mathcal{F}_{0}},\mu,\nu\in L^{2}_{\mathbb{F}^{0}}([0,T];\mathbb{R});\\
    &\mu_{t} = \mathbb{E}[\alpha_{t}^{*}|\mathcal{F}_{t}^{W^{0}}], \quad \nu_{t} = \mathbb{E}[x^{\xi,\alpha^{*}}_{t}|\mathcal{F}^{W^{0}}_{t}].
\end{aligned}\right.
\end{equation} 
\end{problem}

We would like to emphasize that when minimizing the cost functional in the control problem above, $(\mu,\nu)$ is exogenous and
and is not affected by a player's control. Thus, problem \ref{problem} is a standard control problem with an additional consistency condition.
\subsection{Solution of the problem} \label{section 2.3}
In this subsection, we solve the extended MFG problem introduced in previous subsection by stochastic maximum principle. Now, we list all the assumptions on data.\\
\begin{itemize}
    \item [(B1)] The mappings $A_{t},B_{t},R_{t},Q_{t},F_{t}:[0,T] \rightarrow \mathbb{R}$ are measurable and bounded. $\sigma,\sigma_{0}$ are constant vectors and $G$ is a constant. Moreover, $G>0$, $R_{t}>0$ and  $Q_{t}-F_{t}^{2}R^{-1}_{t}\geq 0$ for any $t\in[0,T]$.
    \item[(B2)] The functions $ f(t,x),\ b(t,x),\ l(t,x),\ h(t,x):[0,T]\times \mathbb{R}\rightarrow \mathbb{R}$, $g(x):\mathbb{R}\rightarrow \mathbb{R}$ are measurable and uniformly Lipschitz continuous with respect to $x$. Moreover,\begin{equation*}
        \int_{0}^{T}|f(t,0)|^{2}+|b(t,0)|^{2}+|l(t,0)|^{2}+|h(t,0)|^{2}dt\leq \infty.
    \end{equation*}
     \item [(B3)] There exists some $\varepsilon_0>0$ such that $h(t,\cdot) \in C^1(\mathbb{R};\mathbb{R})$  and $\left|1+h^{\prime}(t,\cdot)\right| \geq \varepsilon_0$ for any $t\in[0,T]$.
\item [(B4)] $f(t,\cdot),\ b(t,\cdot),\ l(t,\cdot),\ h(t,\cdot),\  g(\cdot
) \in C^{2}(\mathbb{R};\mathbb{R})$ and their first and second order derivatives are all bounded for all $t\in[0,T]$. Moreover the derivatives satisfy one of the following conditions :\\
(a)
    \begin{equation*} 
       G(1+g^{\prime}) \geq 0, \quad Q(1+l^{\prime})-F_{t}^{2}R^{-1}_{t}+\frac{(h^{\prime}-q^{\prime})F^{2}_{t}R_{t}^{-1}}{1+h^{\prime}} \geq 0,\quad \frac{(-B_{t}R_{t}^{-1})(B_{t}+b^{\prime})}{1+h^{\prime}}\leq 0.
    \end{equation*}
(b)\begin{equation*} 
       G(1+g^{\prime}) \leq 0, \quad Q(1+l^{\prime}) -F_{t}^{2}R^{-1}_{t}+\frac{(h^{\prime}-q^{\prime})F^{2}_{t}R_{t}^{-1}}{1+h^{\prime}} \leq 0,\quad \frac{(-B_{t}R_{t}^{-1})(B_{t}+b^{\prime})}{1+h^{\prime}}\geq 0.
    \end{equation*}
\end{itemize}

Following standard technique solving MFG problems, we divide our analysis into two steps.

\textbf{Step 1:}
Given $(\mu, \nu)\in L^{2}_{\mathbb{F}^{0}}([0,T];\mathbb{R})\times L^{2}_{\mathbb{F}^{0}}([0,T];\mathbb{R})$,
using the stochastic maximum principle, the minimization problem in Problem \ref{problem}  reduces to solve the following FBSDE
\begin{equation}
\left\{\begin{aligned}
d x_t^{\mu,\nu,\xi} &=  {\left[(A_{t}-B_{t}F_{t}R_{t}^{-1}) x_t^{\mu,\nu,\xi}-B_{t}^2 R_{t}^{-1} y_t^{\mu,\nu,\xi}-B_{t} h\left(t,\mu_t\right)+f\left(t,\nu_t\right)+b\left(t,\mu_t\right)\right] d t } \\
&\quad +\sigma d W_t+\sigma_{0}dW_{t}^{0}, \\
d y_t^{\mu,\nu,\xi} &=  -\left[(A_{t}-B_{t}F_{t}R^{-1}_{t}) y_t^{\mu,\nu,\xi}+(Q_{t}-F_{t}^{2}R^{-1}_{t}) x_t^{\mu,\nu,\xi}+Q_{t} l\left(t,\nu_t\right)+F_{t}q(t,\mu_{t})-F_{t}h(t,\mu_{t})\right] d t\\
&\quad +z_t^{\mu,\nu,\xi} d W_t
+z_{t}^{0,\mu,\nu,\xi}dW_{t}^{0}, \\
x_0^{\mu,\nu,\xi} & =  \xi,\quad  y_T^{\mu,\nu,\xi}=G\left(x_T^{\mu,\nu,\xi}+g\left(\nu_T\right)\right).
\end{aligned}\right.\label{fiXed fbsde}
\end{equation}
Moreover, the corresponding optimal control process is given by
\begin{equation}
\alpha_t^{*}=-R_{t}^{-1} B_{t} y_t^{\mu,\nu,\xi}-R^{-1}_{t}F_{t}x_{t}^{\mu,\nu,\xi}-h\left(t,\mu_t\right). \label{optimal control}
\end{equation}

\textbf{Step 2:}
Taking conditional expectation in \eqref{optimal control} with respect to $\mathcal{F}_{t}^{W^{0}}$, we have
\begin{equation*}
\mathbb{E}[\alpha_{t}^{*}| \mathcal{F}_{t}^{W^{0}}]=-R_{t}^{-1} B_{t}\mathbb{E}[y_{t}^{\mu,\nu,\xi}|\mathcal{F}_{t}^{W^{0}}]-R_{t}^{-1}F_{t}\mathbb{E}[x_{t}^{\mu,\nu,\xi}|\mathcal{F}_{t}^{W^{0}}]-h\left(t,\mu_t\right).
\end{equation*}
The consistency condition in Problem \ref{problem} leads to
\begin{equation}
\mu_t+h\left(t,\mu_t\right)=-R_{t}^{-1} B_{t} \mathbb{E}[y_{t}^{\mu,\nu,\xi}| \mathcal{F}_{t}^{W^{0}}]-R_{t}^{-1}F_{t}\mathbb{E}[x_{t}^{\mu,\nu,\xi}|\mathcal{F}_{t}^{W^{0}}],
\end{equation}
and
\begin{equation}
    \nu_{t} = \mathbb{E}[x_{t}^{\mu,\nu,\xi}|\mathcal{F}_{t}^{W^{0}}].
\end{equation}

Now, we would like to express $\mu_{t}$ as a function of $x_{t}^{\mu,\nu,\xi}$ and $y_{t}^{\mu,\nu,\xi}$.
Applying Assumption (B3), we can use the inverse function theorem to derive that there exists a uniformly Lipschitz continuous function $\rho:[0,T]\times \mathbb{R} \rightarrow \mathbb{R}$ such that
\begin{equation*}
\mu_t=\rho\left(t,-R_{t}^{-1} B_{t} \mathbb{E}[y_{t}^{\mu,\nu,\xi}|\mathcal{F}_{t}^{W^{0}}]-R_{t}^{-1}F_{t}\mathbb{E}[x_{t}^{\mu,\nu,\xi}|\mathcal{F}_{t}^{W^{0}}]\right),
\end{equation*}
where
\begin{equation*}
\rho^{\prime} = \frac{1}{1+h^{\prime}}, \quad 
\left|\rho^{\prime}\right|=\left|\frac{1}{1+h^{\prime}}\right| \leq \frac{1}{\varepsilon_0}.
\end{equation*}
Then the corresponding conditional mean field FBSDE of the extended mean field game is the following FBSDE
\begin{equation}\label{consistency fbsde}
\left\{\begin{aligned}
d x_t^{\xi}= & \left[(A_{t}-B_{t}F_{t}R_{t}^{-1}) x_t^{\xi}-B_{t}^2 R_{t}^{-1} y_t^{\xi}-B_{t} h\left(t,\rho\left(t,-R_{t}^{-1} B_{t} \mathbb{E}[y_{t}^{\xi}|\mathcal{F}_{t}^{W^{0}}]-R_{t}^{-1}F_{t}\mathbb{E}[x_{t}^{\xi}|\mathcal{F}_{t}^{W^{0}}]\right)\right)\right.\\
&\left.+f\left(t,\mathbb{E}[x_{t}^{\xi}|\mathcal{F}_{t}^{W^{0}}]\right)+b\left(t,\rho\left(t,-R_{t}^{-1} B_{t} \mathbb{E}[y_{t}^{\xi}|\mathcal{F}_{t}^{W^{0}}]-R_{t}^{-1}F_{t}\mathbb{E}[x_{t}^{\xi}|\mathcal{F}_{t}^{W^{0}}]\right)\right)\right] d t+\sigma d W_t+\sigma_{0}dW_{t}, \\
d y_t^{\xi}= & -\left[(A_{t}-B_{t}F_{t}R_{t}^{-1}) y_t^{\xi}+(Q_{t}-F_{t}^{2}R^{-1}_{t}) x_t^{\xi}+Q_{t} l\left(t,\mathbb{E}[x_{t}^{\xi}|\mathcal{F}_{t}^{W^{0}}]\right)\right.\\
&\quad + F_{t}q\left(t,\rho\left(t,-R_{t}^{-1} B_{t} \mathbb{E}[y_{t}^{\xi}|\mathcal{F}_{t}^{W^{0}}]-R_{t}^{-1}F_{t}\mathbb{E}[x_{t}^{\xi}|\mathcal{F}_{t}^{W^{0}}]\right)\right)\\
&\left.\quad -F_{t}h\left(t,\rho\left(t,-R_{t}^{-1} B_{t} \mathbb{E}[y_{t}^{\xi}|\mathcal{F}_{t}^{W^{0}}]-R_{t}^{-1}F_{t}\mathbb{E}[x_{t}^{\xi}|\mathcal{F}_{t}^{W^{0}}]\right)\right)\right] d t+z_t^{\xi} d W_t+z_{t}^{0,\xi}dW_{t}^{0}, \\
x_0^{\xi}= & \xi,\quad y_T^{\xi}=G\left(x_T^{\xi}+g(\mathbb{E}[x_{T}^{\xi}|\mathcal{F}^{W^{0}}_{T}])\right).
\end{aligned}\right.
\end{equation}

\begin{Theorem}\label{mf solution}
Under assumptions $(B1)-(B4)$, the conditional mean field FBSDE \eqref{consistency fbsde} has a unique solution $(x^{\xi},y^{\xi},z^{\xi},z^{0,\xi}) \in S_{\mathbb{F}}^{2}([0,T];\mathbb{R})\times S^{2}_{\mathbb{F}}([0,T];\mathbb{R})\times L^{2}_{\mathbb{F}}([0,T];\mathbb{R}^{d})\times L^{2}_{\mathbb{F}}([0,T];\mathbb{R}^{d})$. Moreover,  for any $\xi_{1},\xi_{2}\in L^{2}_{\mathcal {F}_{0}}$, \begin{equation}\label{uniform}
  \left|\mathbb{E}[y_{t}^{\xi_{1}}|\mathcal{F}_{t}^{W^{0}}]-\mathbb{E}[y_{t}^{\xi_{2}}|\mathcal{F}_{t}^{W^{0}}]\right|\leq C\left|\mathbb{E}[x_{t}^{\xi_{1}}|\mathcal{F}_{t}^{W^{0}}]-\mathbb{E}[x_{t}^{\xi_{2}}|\mathcal{F}_{t}^{W^{0}}]\right|,\quad \forall t \in [0,T],
\end{equation}
where C is a constant only depending on $K,T$. Furthermore, 
\begin{equation}
 \alpha_{t}^{*} = -R_{t}^{-1}B_{t}y_{t}^{\xi}-R_{t}^{-1}F_{t}x^{\xi}_{t}-h\left(t, \rho\left(t,-R_{t}^{-1} B_{t} \mathbb{E}[y_{t}^{\xi}|\mathcal{F}_{t}^{W^{0}}]-R_{t}^{-1}F_{t}\mathbb{E}[x_{t}^{\xi}|\mathcal{F}_{t}^{W^{0}}]\right)\right)   
\end{equation}
is an optimal control of problem \ref{problem}.
\end{Theorem}

\begin{proof}
It is easy to verify under assumptions (B1)-(B4), the assumptions (A1)-(A3) hold. Thus it follows from Theorem  \ref{global general theorem} that conditional mean field FBSDE \eqref{consistency fbsde} admits a unique solution and \eqref{uniform} holds. The optimal control statement is a straightforward result of stochastic maximum principle and we refer to \cite{carmona2018probabilistic}.
\end{proof}
\subsection{Master equation for the extended MFGs with common noises}\label{section 2.4}
In this subsection, we will study the well-posedness of the following master equation corresponding to problem \ref{problem}:
\begin{equation}
\left\{\begin{aligned}
    &\partial_t V(t, x, \nu)-\frac{1}{2} B^2_{t} R^{-1}_{t}\left|\partial_x V(t, x, \nu)\right|^2+\frac{1}{2}F_{t}^{2}R^{-1}_{t}x^{2}+\partial_x V(t, x, \nu)\Big[A _{t}x+f(t,\nu) \\
&-B_{t} h\left(t,\rho\left(t,-R^{-1}_{t} B_{t} \mathbb{E}\left[\partial_x V(t, \xi, \nu)\right]-R^{-1}_{t}F_{t}\nu\right)\right)+b\left(t,\rho\left(-R^{-1}_{t} B _{t}\mathbb{E}\left[\partial_x V(t, \xi, \nu)\right]-R^{-1}_{t}F_{t}\nu\right)\right)\Big]\\
&+\partial_\nu V(t, x, \nu)\Big[(A-B_{t}F_{t}R_{t}^{-1}) \nu-B_{t}^2 R_{t}^{-1} \mathbb{E}\left[\partial_x V(t, \xi, \nu)\right]+f(t,\nu)\\
&-B _{t}h\left(t,\rho\left(t,-R_{t}^{-1} B _{t}\mathbb{E}\left[\partial_x V(t, \xi, \nu)\right]-R^{-1}_{t}F_{t}\nu\right)\right)+b\left(t,\rho\left(t,-R_{t}^{-1} B_{t} \mathbb{E}\left[\partial_x V(t, \xi, \nu)\right]-R^{-1}_{t}F_{t}\nu\right)\right)\Big] \\
&+\frac{1}{2} \partial_{x x} V(t, x, \nu)\left(\sigma^2+\sigma_0^2\right) +\frac{1}{2} \partial_{\nu \nu} V(t, x, \nu) \sigma_0^2+\partial_{x \nu} V(t, x, \nu) \sigma_0^2 +\frac{1}{2} Q_{t}(x+l(t,\nu))^2 \\
&+F_{t}x\Big[-F_{t}R_{t}^{-1}x-B_{t}R_{t}^{-1}\partial_{x}V(t,x,\nu)-h\left(t,\rho\left(t,-R^{-1}_{t} B_{t} \mathbb{E}\left[\partial_x V(t, \xi, \nu)\right]-R^{-1}_{t}F_{t}\nu\right)\right)\\
&+q\left(t,\rho\left(t,-R^{-1}_{t} B_{t} \mathbb{E}\left[\partial_x V(t, \xi, \nu)\right]-R^{-1}_{t}F_{t}\nu\right)\right)\Big]=0, \\
&V(T, x, \nu)=\frac{1}{2} G(x+g(\nu))^2
\end{aligned}\right.\label{MP master equation}
\end{equation}
where $\xi \in L_{\mathcal{F}_0}^2$ with $\mathbb{E}[\xi]=\nu$.
We shall analyze the well-posedness of the master equation via the following vectorial master equation
\begin{equation}
\left\{\begin{aligned}
  &\partial_t u(t, x, \nu)+\partial_x u(t, x, \nu)\Big[(A_{t}-B_{t}F_{t}R_{t}^{-1}) x-B_{t}^2 R_{t}^{-1} u(t, x, \nu)\\
&-B h\left(t,\rho\left(t,-R_{t}^{-1} B_{t} \mathbb{E}[u(t, \xi, \nu)]-R^{-1}_{t}F_{t}\nu\right)\right)+f(t,\nu)+b\left(t,\rho\left(t,-R_{t}^{-1} B_{t} \mathbb{E}[u(t, \xi, \nu)]-R^{-1}_{t}F_{t}\nu\right)\right)\Big]\\
&
+\partial_\nu u(t, x, \nu)\Big[(A_{t}-B_{t}F_{t}R_{t}^{-1}) \nu-B^2 _{t}R_{t}^{-1} \mathbb{E}[u(t, \xi, \nu)]\\
&-B _{t}h\left(t,\rho\left(t,-R_{t}^{-1} B_{t} \mathbb{E}[u(t, \xi, \nu)]-R^{-1}_{t}F_{t}\nu\right)\right)+f(t,\nu)+b\left(t,\rho\left(t,-R_{t}^{-1} B_{t} \mathbb{E}[u(t, \xi, \nu)]-R^{-1}_{t}F_{t}\nu\right)\right)\Big]\\
&+\frac{1}{2} \partial_{x x} u(t, x, \nu)\left(\sigma^2+\sigma_0^2\right)+\frac{1}{2} \partial_{\nu \nu} u(t, x, \nu) \sigma_0^2+\partial_{x \nu} u(t, x, \nu) \sigma_0^2 +(A_{t}-B_{t}F_{t}R_{t}^{-1}) u(t, x, \nu)\\
&+(Q_{t}-F^{2}_{t}R_{t}^{-1})x+Q_{t}l(t,\nu)
+F_{t}q\left(t,\rho\left(t,-R_{t}^{-1} B_{t} \mathbb{E}[u(t, \xi, \nu)]-R^{-1}_{t}F_{t}\nu\right)\right)\\
&-F_{t}h\left(t,\rho\left(t,-R_{t}^{-1} B_{t} \mathbb{E}[u(t, \xi, \nu)]-R^{-1}_{t}F_{t}\nu\right)\right)=0,\\
&u(T, x, \nu)=G(x+g(\nu)). \end{aligned}
\right.\label{master equation}
\end{equation}

Actually, the solution $u$ serves as the decoupling field of $\eqref{consistency fbsde}$. Noting that the extended MFGs master equation \eqref{master equation} is a special case of the master equation \eqref{general master equation} studied in section \ref{section 3.2}, the following Theorem is a straightforward consequence of Theorem \ref{master theorem} and the proof will be omitted.
\begin{Theorem}\label{u master classical}
 Let Assumptions $(B1)-(B4)$ hold, the vectorial master equation \eqref{master equation} admits a unique classical solution $u$ 
 with bounded $\partial_{x}u,\partial_{\nu}u, \partial_{\nu\nu}u$.
\end{Theorem}

 Based on Theorem \ref{u master classical}, applying similar technique as in  \cite[Theorem 3.8]{li2023linear}, we obtain the following theorem.
\begin{Theorem}
    Let Assumptions $(B1)-(B4)$ hold. Then the master equation \eqref{MP master equation} admits a unique classical solution $V$ with bounded $\partial_{x x} V, \partial_{x \nu} V, \partial_{\nu \nu} V$.
\end{Theorem}
\begin{remark} \label{comparison remark}
  Let us summarize the main novelties compared with \cite{li2023linear} as follows:
\begin{itemize}
    \item In their model, there is no cross term of state and control in cost functional, that is $F_{t} = 0$ and the drift coefficient of state dynamics is $AX_{t}+B\alpha_{t}+f(\nu_{t})+b(\mu_{t})$, where $A,B$ are constants. By contrast, we consider a more general form of state dynamics and cost functional.
    \item They require the boundedness of the coefficient $(f,b,l,h,g)$ and the non-degeneracy of the common noise. We remove the boundedness and non-degeneracy assumptions by adding some monotonicity conditions.
\end{itemize}  
\end{remark}
\section{Solvability of a type of conditional mean field  FBSDEs and related master equations}\label{section 3}
\subsection{Conditional mean field FBSDEs}
Let $(\Omega, \mathcal{F}, \mathbb{F}, \mathbb{P})$ be a complete filtered probability space which can support two independent d-dimensional Brownian motions: $W$ and $W^0$.  We let $\mathbb{F}:=\left\{\mathcal{F}_t\right\}_{t \in[0, T]}$, where $\mathcal{F}_t:=\mathcal{F}_t^W \vee \mathcal{F}_t^{W^0} \vee \mathcal{F}_0$, and let $\mathbb{P}$ have no atom in $\mathcal{F}_0$ so it can support any measure on $\mathbb{R}$ with a finite second moment. For simplicity, We denote $\mathbb{F}^0:=\left\{\mathcal{F}_t^{W^0}\right\}_{t \in[0, T]}$ and $\mathbb{F}^{1}:= \left\{\mathcal{F}^{W}_{t}\right\}_{t\in[0,T]}$. In this subsection, we prove existence and uniqueness of solution for  conditional mean field FBSDE
\begin{equation}\label{conditional fbsde}
    \left\{
    \begin{aligned}
dX_{t} &= \xi+\int_{0}^{t}\Big[b_{1}(s)X_{s}+b_{2}(s)Y_{s}+b_{0}\left(s,E[X_{s}|\mathcal{F}^{W^{0}}_{s}],E[X_{s}|\mathcal{F}^{W^{0}}_{s}]\right)\Big]ds+\int_{0}^{t}\sigma(s)dW_{s}\\
&\quad +\int_{0}^{t}\Big[\sigma_{1}(s)X_{s}+ Y_{s}^{\mathsf{T}}\sigma_{2}(s)+\sigma_{0}\left
(s,E[X_{s}|\mathcal{F}^{W^{0}}_{s}],E[Y_{s}|\mathcal{F}^{W^{0}}_{s}]\right) \Big] dW_{s}^{0},\\
dY_{t} &=h_{1}X_{T}+h_{2}(\mathbb{E}[X_{T}|\mathcal{F}_{T}^{W^{0}}])+\int_{t}^{T} \Big[f_{1}(s)X_{s}+f_{2}(s)Y_{s}+f_{0}\left(s,E[X_{s}|\mathcal{F}^{W^{0}}_{s}],E[X_{s}|\mathcal{F}^{W^{0}}_{s}]\right)\Big]ds\\
&\quad -\int_{t}^{T}Z_{s}dW_{s}-\int_{t}^{T}Z^{0}_{s}dW_{s}^{0},
\end{aligned}\right.
\end{equation}
where $W$ and $W^{0}$ are two independent d-dimensional Brownian motion and $\xi \in L^{2}_{\mathcal{F}_{0}}$. We assume the system's coefficients satisfy the following assumptions and let $K$ be a positive constant.
\begin{itemize}
\item[\textbf{(A1)}]
(i) The mappings $b_{1}: [0, T] \rightarrow  \mathbb{R}$, $b_{2},f_{1}: [0, T] \rightarrow  \mathbb{R}^{n}$, $f_{2}: [0, T] \rightarrow  \mathbb{R}^{n\times n}, \sigma_{1}:[0,T]\rightarrow \mathbb{R}^{d}, \sigma_{2}: [0,T]\rightarrow \mathbb{R}^{n\times d}$ are measurable and bounded by $K$, the mapping $\sigma:\Omega \times [0,T] \rightarrow \mathbb{R}^{d}$ is $\mathcal{F}_{t}^{W}$-progressively measurable.\\
(ii) The mappings $ b_{0}(t,\bar{x},\bar{y}):\Omega\times[0,T]\times \mathbb{R} \times \mathbb{R}^{n}\rightarrow \mathbb{R}$, $f_{0}(t,\bar{x},\bar{y}):\Omega\times[0,T]\times \mathbb{R} \times \mathbb{R}^{n}\rightarrow \mathbb{R}^{n}$ and $\sigma_{0}(t,\bar{x},\bar{y}):\Omega \times [0,T] \times \mathbb{R} \times \mathbb{R}^{n} \rightarrow \mathbb{R}^{d} $ are $\mathcal{F}^{W^{0}}_{t}$-progressively measurable and uniformly Lipschitz continuous with respect to $\bar{x}$ and $\bar{y}$, i.e., 
 \begin{equation*}
      \begin{aligned}
 \left|b_{0}(t,\bar{x}, \bar{y})-b_{0}\left(t,\bar{x}^{\prime},\bar{y}^{\prime}\right)\right| & \leq K\left(\left|\bar{x}-\bar{x}^{\prime}\right|+ \left|\bar{y}-\bar{y}^{\prime}\right|\right),\\
 \left|f_{0}(t,\bar{x}, \bar{y})-f_{0}\left(t,\bar{x}^{\prime},\bar{y}^{\prime}\right)\right| & \leq K\left(\left|\bar{x}-\bar{x}^{\prime}\right|+ \left|\bar{y}-\bar{y}^{\prime}\right|\right),\\
 \left|\sigma_{0}(t,\bar{x}, \bar{y})-\sigma_{0}\left(t,\bar{x}^{\prime},\bar{y}^{\prime}\right)\right| & \leq K\left(\left|\bar{x}-\bar{x}^{\prime}\right|+ \left|\bar{y}-\bar{y}^{\prime}\right|\right).
\end{aligned}
    \end{equation*}
    Moreover,  $b_{0}(t,0,0)\in L^{2}_{\mathbb{F}^{0}}([0,T];\mathbb{R}),f_{0}(t,0,0)\in L_{\mathbb{F}^{0}}^{2}([0,T];\mathbb{R}^{n}),\sigma_{0}(t,0,0)\in L^{2}_{\mathbb{F}^{0}}([0,T];\mathbb{R}^{d}),\sigma(t)\in L^{2}_{\mathbb{F}^{1}}([0,T];\mathbb{R}^{d})$.\\
(iii) $h_{1}\in\mathbb{R}$ is bounded by $K$ and $h_{2}:\Omega \times \mathbb{R}\rightarrow \mathbb{R}^{n}$ is $\mathcal{F}_{T}^{W^0}$-measurable and uniformly Lipschitz continuous, i.e., 
    \begin{equation*}
        \left| h_{2}(\bar{x})-h_{2}(\bar{x}^{\prime}) \right|\leqslant K \left|\bar{x}-\bar{x}^{\prime} \right|.
    \end{equation*}
Moreover, $h_{2}(0) \in L^{2}_{\mathcal{F}_{T}^{W^{0}}}$.
 \end{itemize}
 
  Next, we give the following definitions and notations. For ease of notations,  for $y_{1},y_{2}\in \mathbb{R}^{n}$, we denote
\begin{equation*}
(y_{1}^{(1,i)},y_{2}^{(i+1,n)}):=(y_{1}^{1},y_{1}^{2},\cdots,y_{1}^{i},y_{2}^{i+1},\cdots,y_{2}^{n}).
\end{equation*}
 For $\left(\bar{x}_{1}, \bar{y}_{1}\right), \left(\bar{x}_{2}, \bar{y}_{2}\right)\in \mathbb{R} \times \mathbb{R}^{n}$, let $\theta_{1}:= (\bar{x}_{1},\bar{y}_{1}), \theta_{2}:= (\bar{x}_{2},\bar{y}_{2})$, and for $i,j=1,2,\cdots,n$, denote
\begin{equation}
	\begin{aligned}
		h_{2}^{i}(\bar{x}_{1},\bar{x}_{2}) &\triangleq \frac{h^{i}_{2}(\bar{x}_{1})-h^{i}_{2}(\bar{x}_{2})}{\bar{x}_{1}-\bar{x}_{2}},\\
	b_{3}(t,\theta_{1},\theta_{2})&\triangleq \frac{b_{0}(t,\bar{x}_{1},\bar{y}_{1})-b_{0}(t,\bar{x}_{2},\bar{y}_{1})}{\bar{x}_{1}-\bar{x}_{2}},\\
		f_{3}^{i}(t,\theta_{1},\theta_{2}) &\triangleq \frac{f^{i}_{0}(t,\bar{x}_{1},\bar{y}_{1})-f^{i}(t,\bar{x}_{2},\bar{y}_{1})}{\bar{x}_{1}-\bar{x}_{2}},\\
	b_{4}^{j}(t,\theta_{1},\theta_{2}) &\triangleq \frac{b_{0}(t,\bar{x}_{2},\bar{y}_{2}^{(1,j-1)},\bar{y}_{1}^{(j,n)})-b_{0}(t,\bar{x}_{2},\bar{y}_{2}^{(1,j)},\bar{y}_{1}^{(j+1,n)})}{\bar{y}_{1}^{j}-\bar{y}_{2}^{j}},\\
	f_{4}^{ij}(t,\theta_{1},\theta_{2})&\triangleq \frac{f^{i}_{0}(t,\bar{x}_{2},\bar{y}_{2}^{(1,j-1)},\bar{y}_{1}^{(j,n)})-f^{i}_{0}(t,\bar{x}_{2},\bar{y}_{2}^{(1,j)},\bar{y}_{1}^{(j+1,n)})}{\bar{y}_{1}^{j}-\bar{y}_{2}^{j}},
		\end{aligned} \label{quadratic notation}
\end{equation}
and $b_{4}(t,\theta_{1},\theta_{2}) = (b_{4}^{1},b_{4}^{2},\cdots,b_{4}^{n})(t,\theta_{1},\theta_{2})$, $f_{4}^{i}(t,\theta_{1},\theta_{2}) = (f_{j}^{i1},f_{j}^{i2},\cdots,f_{j}^{in})(t,\theta_{1},\theta_{2})$.

Using these notations, we introduce the following monotonicity conditions.
\begin{itemize}
\item[\textbf{(A2)}]
 For $1\leq i,j \leq n$, $j\neq i$ and  $t\in[0,T] $, one of the following two cases holds:
    \begin{itemize}
        \item[(i)] 
\begin{equation}
  f_{1}(t) \geq 0, \quad h_{1}\geq 0,\quad b_{2}(t)\leq 0 \text{ and } 
    f_{2}^{ij}(t)\geq 0.
\end{equation} 
\item[(ii)] 
\begin{equation}
  f_{1}(t) \leq 0, \quad h_{1}\leq 0,\quad b_{2}(t)\geq 0 \text{ and } 
    f_{2}^{ij}(t)\geq 0.
\end{equation} 
\end{itemize}
\item[\textbf{(A3)}]
For $1\leq i,j \leq n$, $i\neq j$, $t\in[0,T] $ and  any $\theta_{1}=(\bar{x}_{1},\bar{y}_{1}),\theta_{2}=(\bar{x}_{2},\bar{y}_{2}) \in \mathbb{R}\times \mathbb{R}^{n}$, one of the following two cases holds:
\begin{itemize}
    \item [(i)] 
$f_{1}(t)+f_{3}(t,\theta_{1},\theta_{2}) \geq 0, \ h_{1}+h_{2}(\bar{x}_{1},\bar{x}_{2})\geq 0,\ b_{2}(t)+b_{4}(t,\theta_{1},\theta_{2})\leq 0$ , $f_{2}^{ij}(t)+f_{4}^{ij}(t,\theta_{1},\theta_{2})\geq 0$.
\item [(ii)]
$f_{1}(t)+f_{3}(t,\theta_{1},\theta_{2}) \leq 0, \ h_{1}+h_{2}(\bar{x}_{1},\bar{x}_{2})\leq 0,\ b_{2}(t)+b_{4}(t,\theta_{1},\theta_{2})\geq 0$ , $f_{2}^{ij}(t)+f_{4}^{ij}(t,\theta_{1},\theta_{2})\geq 0$.
\end{itemize}
\end{itemize}

The following theorem is the main result of this subsection, which gives the existence and uniqueness of solution for conditional mean field FBSDE \eqref{conditional fbsde}.
\begin{Theorem}\label{global general theorem}
 Under assumptions $ (A1)-(A3)$, the conditional mean field FBSDE \eqref{conditional fbsde} admits a unique solution $(X,Y,Z,Z^{0}) \in S^{2}_{\mathbb{F}}([0,T];\mathbb{R}^{d})\times S^{2}_{\mathbb{F}}([0,T];\mathbb{R}^{d})\times L^{2}_{\mathbb{F}}([0,T];\mathbb{R}^{n\times d})\times L^{2}_{\mathbb{F}}([0,T];\mathbb{R}^{n\times d})$.
\end{Theorem}

The strategy we use below is to recast the stochastic system \eqref{conditional fbsde} into a well-posed fixed point problem over conditional expectations. The first step is to use $(\mathbb{E}[X_{t}|\mathcal{F}^{W^{0}}_{t}],\mathbb{E}[Y_{t}|\mathcal{F}_{t}^{W^{0}}])_{0\leq t\leq T}$ as an input and then solve \eqref{conditional fbsde} as a classical FBSDE. Therefore, for any given $(\mu,\nu) \in L^{2}_{\mathbb{F}^{0}}([0,T];\mathbb{R}) \times L^{2}_{\mathbb{F}^{0}}([0,T];\mathbb{R}^{n})$, we first consider the following FBSDE : 
\begin{equation}
\left\{
\begin{aligned}
     dX_{t}^{\xi,\mu,\nu}&=\xi+\int_{0}^{t} \Big[b_{1}(s)X_{s}^{\xi,\mu,\nu}+b_{2}(s)Y_{s}^{\xi,\mu,\nu}+b_{0}(s,\mu_{s},\nu_{s})\Big]dt+\int_{0}^{t}\sigma (s)dW_{s}\\
     &\quad +\int_{0}^{t}\Big[\sigma_{1}(s)X_{s}^{\xi,\mu,\nu}+(Y_{s}^{\xi,\mu,\nu})^{\mathsf{T}}\sigma_{2}(s)+\sigma_{0}(s,\mu_{s},\nu_{s}) \Big]dW_{s}^{0},\\
dY_{t}^{\xi,\mu,\nu} &=h_{1}X_{T}^{\xi,\mu,\nu}+h_{2}(\mu_{T})+\int_{t}^{T} \Big[f_{1}(s)X_{s}^{\xi,\mu,\nu}+f_{2}(s)Y_{s}^{\xi,\mu,\nu}+f_{0}(s,\mu_{s},\nu_{s})\Big]ds\\
&\quad -\int_{t}^{T}Z_{s}^{\xi,\mu,\nu}dW_{s}-\int_{t}^{T}Z_{s}^{0,\xi,\mu,\nu}dW_{s}^{0}.
    \end{aligned}\right. \label{general fixed fbsde}
\end{equation}

Now, we would like to show the well-posedness of FBSDE \eqref{general fixed fbsde} on $[0,T]$. Further, we shall give a representation result of the solution of FBSDE \eqref{general fixed fbsde}. To this end, we introduce the following Riccati equation
 \begin{equation}
     \left\{
     \begin{aligned}
         &dP_{t}=-(b_{2}(t) P_{t})P_{t}-f_{2}(t) P_{t}-b_{1}(t)P_{t}-f_{1}(t),\\
         &P_{T} = h_{1}.
     \end{aligned}
     \right.\label{p equation}
 \end{equation}
 For the solvability of above Riccati equation, we recall the following Lemma from \cite[Lemma 4.1]{hua2023well}.
\begin{lemma}\label{p lemma}
Under assumptions $(A1)$ and $ (A2)$, the Riccati equation \eqref{p equation}  admits a unique solution on $[0,T]$ such that for any $t\in[0,T]$,
 \begin{equation*}
    \left|P_{t}\right| \leq C,
 \end{equation*}
 where $C$ is a constant depending only on $n,K,T$.
\end{lemma}

Next, we introduce the following backward stochastic differential equation (BSDE)
 \begin{equation}\label{BSDE}
 \left\{
 \begin{aligned}
    d\varphi_{t} ^{\xi,\mu,\nu}&=-\left[ f_{2}(t)\varphi_{t}^{\xi,\mu,\nu}+(b_{2}(t)\varphi_{t}^{\xi,\mu,\nu})P_{t} +f_{0}(t,\mu_{t},\nu_{t})+P_{t}b_{0}(t,\mu_{t},\nu_{t})\right] dt +\Gamma_{t}^{\xi,\mu,\nu}dW_{t}^{0},   \\
    \varphi_{T}^{\xi,\mu,\nu} &= h_{2}(\mu_{T}). 
 \end{aligned}
  \right.
 \end{equation}
\begin{lemma}
Suppose assumptions $(A1)$ and $  (A2)$ hold. Let $\xi\in L^{2}_{\mathcal{F}_{0}}$ and $(\mu,\nu) \in L^{2}_{\mathbb{F}^{0}}([0,T];\mathbb{R}) \times L^{2}_{\mathbb{F}^{0}}([0,T];\mathbb{R}^n)$.\\
\begin{itemize}
    \item [(i)]The BSDE \eqref{BSDE} admits a unique solution $(\varphi_{t}^{\xi,\mu,\nu}, \Gamma_{t}^{\xi,\mu,\nu})$ on $[0,T]$.
    \item[(ii)] Given the following stochastic differential equation (SDE) 
\begin{equation}
\left\{
    \begin{aligned}
      dX_{t}^{\xi,\mu,\nu}&=\left[ b_{1}X_{t}^{\xi,\mu,\nu}+b_{2}(P_{t}X_{t}^{\xi,\mu,\nu}+\varphi_{t}^{\xi,\mu,\nu})+b_{0}(t,\mu_{t},\nu_{t})\right]dt+\sigma(t)dW_{t}\\
      &\quad +\Big[\sigma_{1}(t)X_{t}^{\xi,\mu,\nu}+(P_t X_t^{\xi,\mu,\nu}+\varphi_t^{\xi,\mu,\nu})^{\mathsf{T}}\sigma_{2}(t)+\sigma_{0}(t,\mu_{t},\nu_{t}) \Big]dW_{t}^{0},\\
      X_{0}&=\xi.
\end{aligned}\right.\label{SDE}
\end{equation}
SDE \eqref{SDE} is well-posed on $[0,T]$.
\item[(iii)] Given $X_{t}^{\xi,\mu,\nu}$ in (ii), the BSDE of \eqref{general fixed fbsde} admits a unique solution$ (Y_{t}^{\xi,\mu,\nu},Z_{t}^{\xi,\mu,\nu},Z_{t}^{0,\xi})$ on $[0,T]$ where
  $Y_t^{\xi,\mu,\nu}=P_t X_t^{\xi,\mu,\nu}+\varphi_t^{\xi,\mu,\nu}, Z_t^{\xi,\mu,\nu}=P_t\sigma(t)$ and $Z_t^{0,\xi,\mu,\nu}=P_t \big[\sigma_{1}(t)X_{t}^{\xi,\mu,\nu}+(P_t X_t^{\xi,\mu,\nu}+\varphi_t^{\xi,\mu,\nu})^{\mathsf{T}}\sigma_{2}(t)+\sigma^{0}(t,\mu_{t},\nu_{t}) \big]+\Gamma_t^{\xi,\mu,\nu}$. Therefore, $\left(X^{\xi,\mu,\nu}, Y^{\xi,\mu,\nu}, Z^{\xi,\mu,\nu}, Z^{0,\xi,\mu,\nu}\right)$ is the unique solution to FBSDE \eqref{general fixed fbsde}.
\end{itemize}\label{fix representation}
\end{lemma}
\begin{proof}
From Lemma \ref{p lemma} and assumption (A1), we note by the standard  BSDE theory, BSDE \eqref{BSDE} admits a unique solution $(\varphi^{\xi,\mu,\nu}, \Gamma^{\xi,\mu,\nu})$. Moreover, from standard SDE theory, the well-posedness of SDE \eqref{SDE} on $[0,T]$ is obtained. Now, define
\begin{equation*}
Y_{t}^{\xi,\mu,\nu}:=P_{t}X^{\xi,\mu,\nu}_{t}+\varphi^{\xi,\mu,\nu}_{t}, \quad Z_t^{\xi,\mu,\nu}:=P_t\sigma(t),
\end{equation*}
and 
\begin{equation*}
   Z_t^{0,\xi,\mu,\nu}:=P_t \left[\sigma_{1}(t)X_{t}^{\xi,\mu,\nu}+(P_t X_t^{\xi,\mu,\nu}+\varphi_t^{\xi,\mu,\nu})^{\mathsf{T}}\sigma_{2}(t)+\sigma_{0}(t,\mu_{t},\nu_{t}) \right]+\Gamma_t^{\xi,\mu,\nu}.
\end{equation*}
Note that $Y_{T}^{\xi,\mu,\nu} = P_{T}X_{T}^{\xi,\mu,\nu}+\varphi_{T}^{\xi,\mu,\nu}=P_{T}X_{T}^{\xi,\mu,\nu}+h_{2}(\mu_{T})$. Further, applying It\^{o}'s formula to $Y_{t}^{\xi,\mu,\nu}$, it holds that
\begin{equation*}
\begin{aligned}
    dY_{t}^{\xi,\mu,\nu} &= \Big[-(b_{2}(t) P_{t})P_{t}-f_{2}(t) P_{t}-b_{1}(t)P_{t}-f_{1}(t)\Big]X_{t}^{\xi,\mu,\nu}dt\\
    & \quad +P_{t}\left[ b_{1}(t)X_{t}^{\xi,\mu,\nu}+b_{2}(t)(P_{t}X_{t}^{\xi,\mu,\nu}+\varphi_{t}^{\xi,\mu,\nu})+b_{0}(t,\mu_{t},\nu_{t})\right]dt\\ 
    &\quad +P_{t}\left[\sigma(t)dW_{t}+\left(\sigma_{1}(t)X_{t}^{\xi,\mu,\nu}+(P_t X_t^{\xi,\mu,\nu}+\varphi_t^{\xi,\mu,\nu})^{\mathsf{T}}\sigma_{2}(t)+\sigma_{0}(t,\mu_{t},\nu_{t}) \right)dW_{t}^{0}\right]\\
    &\quad -\left[ f_{2}(t)\varphi_{t}^{\xi,\mu,\nu}+(b_{2}(t)\varphi_{t}^{\xi,\mu,\nu})P_{t} +f_{0}(t,\mu_{t},\nu_{t})+P_{t}b_{0}(t,\mu_{t},\nu_{t})\right] dt +\Gamma_{t}^{\xi,\mu,\nu}dW_{t}^{0}\\
    &=-\left[f_{1}(t)X_{t}^{\xi,\mu,\nu}+f_{2}(t)Y_{t}^{\xi,\mu,\nu}+f_{0}(t,\mu_{t},\nu_{t})\right]dt+Z_{t}^{\xi,\mu,\nu}dW_{t}+Z_{t}^{0,\xi,\mu,\nu}dW_{t}^{0}.
\end{aligned}
\end{equation*}
Moreover, recalling $Y_{t}^{\xi,\mu,\nu}=P_{t}X_{t}^{\xi,\mu,\nu}+\varphi^{\xi,\mu,\nu}_{t}$, we observe that $X_{t}^{\xi,\mu,\nu}$ satisfies the forward stochastic differential equation in \eqref{general fixed fbsde} on $[0,T]$. Therefore, we verify that $(X_{t}^{\xi,\mu,\nu},Y_{t}^{\xi,\mu,\nu},Z_{t}^{\xi,\mu,\nu},Z_{t}^{0,\xi,\mu,\nu})$ is a strong solution to FBSDE \eqref{general fixed fbsde}. The uniqueness of the solution to \eqref{general fixed fbsde} on $[0,T]$ follows from the uniqueness property stated in \cite{hua2022unified}.
\end{proof}

 Until now, the results are based on the given stochastic measure flow $(\mu,\nu) \in L^{2}_{\mathbb{F}^{0}}([0,T];\mathbb{R}) \times L^{2}_{\mathbb{F}^{0}}([0,T];\mathbb{R}^n)$. Now we introduce following map:
\begin{equation}
\begin{aligned}
  \mathcal{M}: L^{2}_{\mathbb{F}^{0}}([0,T];\mathbb{R}) \times L^{2}_{\mathbb{F}^{0}}([0,T];\mathbb{R}^n) &\rightarrow L^{2}_{\mathbb{F}^{0}}([0,T];\mathbb{R}) \times L^{2}_{\mathbb{F}^{0}}([0,T];\mathbb{R}^n),\\ (\mu_{t},\nu_{t}) &\mapsto (\mathbb{E}[X_{t}^{\xi,\mu,\nu}|\mathcal{F}_{t}^{W^{0}}],\mathbb{E}[Y_{t}^{\xi,\mu,\nu}|\mathcal{F}_{t}^{W^{0}}]).
  \end{aligned}
  \end{equation}

  By Lemma \ref{fix representation}, map $\mathcal{M}$ is well-defined and our goal is to find a fixed point for this map, in particular, the fixed point is unique. 
We can show that for any $s \in[0, t]$, $ \mathbb{E}[X_{s}^{\xi,\mu,\nu} | \mathcal{F}_t^{W^{0}}]=\mathbb{E}[X_{s}^{\xi,\mu,\nu} | \mathcal{F}_s^{W^{0}}]$. In fact, let $\mathcal{F}_{s, t}^{W^{0}}:=$ $\sigma\left\{W_r^0-W_s^0, s \leq r \leq t\right\}$. For all $s \in[0, t]$, it holds that $\mathcal{F}_t^{W^{0}}=\mathcal{F}_s^{W^{0}} \vee \mathcal{F}_{s, t}^{W^{0}}$, where $\mathcal{F}_s^{W^{0}}$ and $\mathcal{F}_{s, t}^{W^{0}}$ are independent. Noting that $X_{s}^{\xi,\mu,\nu} $ is independent of $\mathcal{F}_{s, t}^{W^{0}}$, we conclude $\mathbb{E}[X_{s}^{\xi,\mu,\nu} | \mathcal{F}_t^{W^{0}}]=\mathbb{E}[X_{s}^{\xi,\mu,\nu} | \mathcal{F}_s^{W^{0}} \vee \mathcal{F}_{s, t}^{W^{0}}]=\mathbb{E}[X_{s} ^{\xi,\mu,\nu}| \mathcal{F}_s^{W^{0}}]$.

Taking conditional expectation in the forward equation of \eqref{general fixed fbsde} with respect to $\mathcal{F}_t^{W^{0}}$, we have
\begin{equation}\label{expectation SDE}
\begin{aligned}
   \mathbb{E}[X_{t}^{\xi,\mu,\nu}| \mathcal{F}_{t}^{W^{0}}]&= 
\mathbb{E}[\xi]+\int_0^t\left[b_{1}(s) \mathbb{E}[X_{s}^{\xi,\mu,\nu} | \mathcal{F}_s^{W^{0}}]+b_{2}(s)\mathbb{E}[Y_{s}^{\xi,\mu,\nu} | \mathcal{F}_s^{W^{0}}] +b_{0}(s,\mu_{s},\nu_{s})\right]ds\\
&\quad +\int_{0}^{t}\left[\sigma_{1}(s)\mathbb{E}[X_{s}^{\xi,\mu,\nu} | \mathcal{F}_s^{W^{0}}]+\mathbb{E}[Y_{s}^{\xi,\mu,\nu} | \mathcal{F}_s^{W^{0}}]^{\mathsf{T}}\sigma_{2}(s)+\sigma_{0}(s,\mu_{s},\nu_{s}) \right]dW_{s}^{0}.
\end{aligned}
\end{equation}
\indent Similarly, recalling that $Z_{t}^{\xi,\mu,\nu} = P_{t}\sigma(t)$ from Lemma \ref{fix representation}, $P_{t}$ is deterministic and $\sigma(t)$ is $\mathcal{F}^{W}_{t}$ measurable, hence $ \mathbb{E}[\int_{0}^{t}Z_{s}^{\xi,\mu,\nu} dW_{s}| \mathcal{F}_{t}^{W^{0}}] = 0 $. Taking conditional expectation for the dynamics of $Y_{t}^{\xi,\mu,\nu} $ in \eqref{general fixed fbsde} with respect to $\mathcal{F}_{t}^{W^{0}}$, we get
\begin{equation}\label{expectation BSDE}
\begin{aligned}
\mathbb{E}[Y_{t}^{\xi,\mu,\nu} | \mathcal{F}_{t}^{W^{0}}]&=  \mathbb{E}[Y_{0}^{\xi,\mu,\nu} ]-\int_0^t\left[f_{1}(s) \mathbb{E}[X_{s}^{\xi,\mu,\nu}  | \mathcal{F}_s^{W^{0}}]+f_{2}(s)\mathbb{E}[Y_{s} ^{\xi,\mu,\nu} | \mathcal{F}_s^{W^{0}}] +f_{0}(s,\mu_{s},\nu_{s})\right]ds\\
&\quad+\int_{0}^{t}Z_{s}^{0,\xi,\mu,\nu}dW_{s}^{0}.    
\end{aligned}
\end{equation}

Combining
\eqref{expectation SDE} and \eqref{expectation BSDE}, we obtain that
 $(\mathbb{E}[X_{t}^{\xi,\mu,\nu}|\mathcal{F}_{t}^{W^{0}}],\mathbb{E}[Y_{t}^{\xi,\mu,\nu}|\mathcal{F}_{t}^{W^{0}}])$ satisfies the following classical FBSDE
\begin{equation}
\left\{
\begin{aligned}
   d\mathbb{E}[X_{t}^{\xi,\mu,\nu}|\mathcal{F}_{t}^{W^{0}}] &=\left[b_{1}(t)\mathbb{E}[X_{t}^{\xi,\mu,\nu}|\mathcal{F}_{t}^{W^{0}}]+b_{2}(t)\mathbb{E}[Y_{t}^{\xi,\mu,\nu}|\mathcal{F}_{t}^{W^{0}}]+b_{0}(t,\mu_{t},\nu_{t})\right]dt\\
   &\quad+\Big[\sigma_{1}(t)\mathbb{E}[X_{t}^{\xi,\mu,\nu} | \mathcal{F}_t^{W^{0}}]
   +\mathbb{E}[Y_{t}^{\xi,\mu,\nu} | \mathcal{F}_t^{W^{0}}]^{\mathsf{T}}\sigma_{2}(t)+\sigma_{0}(t,\mu_{t},\nu_{t}) \Big]dW_{t}^{0},\\
   d\mathbb{E}[Y_{t}^{\xi,\mu,\nu}|\mathcal{F}_{t}^{W^{0}}] &= -\Big[f_{1}(t)\mathbb{E}[X_{t}^{\xi,\mu,\nu}|\mathcal{F}_{t}^{W^{0}}]+f_{2}(t)\mathbb{E}[Y_{t}^{\xi,\mu,\nu}|\mathcal{F}_{t}^{W^{0}}]+f_{0}(t,\mu_{t},\nu_{t})\Big]dt+Z_{t}^{0,\xi,\mu,\nu}dW_{t}^{0},\\
     \mathbb{E}[X_{0}^{\xi,\mu,\nu}|\mathcal{F}_{0}^{W^{0}}] &= \mathbb{E}[\xi],\quad \mathbb{E}[Y_{T}^{\xi,\mu,\nu}|\mathcal{F}_{T}^{W^{0}}] = h_{1}\mathbb{E}[X_{T}^{\xi,\mu,\nu}|\mathcal{F}_{T}^{W^{0}}]+h_{2}(\mu_{T}).
\end{aligned}\right.\label{fixed conditional FBSDE}
\end{equation}

Therefore finding the fixed point of the map $\mathcal{M}$ is converted to solve the corresponding FBSDE by replacing $(\mu_t,\nu_t)$ in \eqref{fixed conditional FBSDE} with $(\mathbb{E}[X_{t}^{\xi,\mu,\nu}|\mathcal{F}_{t}^{W^{0}}],\mathbb{E}[Y_{t}^{\xi,\mu,\nu}|\mathcal{F}_{t}^{W^{0}}])$, which is equivalent with the following FBSDE
\begin{equation}
\left\{
\begin{aligned}
   d\mathbb{E}[X_{t}^{\xi}|\mathcal{F}_{t}^{W^{0}}] &=\Big[ b_{1}(t)\mathbb{E}[X_{t}^{\xi}|\mathcal{F}_{t}^{W^{0}}]+b_{2}(t)\mathbb{E}[Y_{t}^{\xi}|\mathcal{F}_{t}^{W^{0}}]+b_{0}(t,\mathbb{E}[X_{t}^{\xi}|\mathcal{F}_{t}^{W^{0}}],\mathbb{E}[Y_{t}^{\xi}|\mathcal{F}_{t}^{W^{0}}])\Big]dt\\
   &\quad +\Big[\sigma_{1}(t)\mathbb{E}[X_{t}^{\xi}| \mathcal{F}_t^{W^{0}}]+\mathbb{E}[Y_{t} ^{\xi}| \mathcal{F}_t^{W^{0}}]^{\mathsf{T}}\sigma_{2}(t)+\sigma_{0}(t,\mathbb{E}[X_{t}^{\xi}|\mathcal{F}_{t}^{W^{0}}],\mathbb{E}[Y_{t}^{\xi}|\mathcal{F}_{t}^{W^{0}}]) \Big]dW_{t}^{0},\\
   d\mathbb{E}[Y_{t}^{\xi}|\mathcal{F}_{t}^{W^{0}}] &= -\Big[f_{1}(t)\mathbb{E}[X_{t}^{\xi}|\mathcal{F}_{t}^{W^{0}}]+f_{2}(t)\mathbb{E}[Y_{t}^{\xi}|\mathcal{F}_{t}^{W^{0}}]+f_{0}(t,\mathbb{E}[X_{t}^{\xi}|\mathcal{F}_{t}^{W^{0}}],\mathbb{E}[Y_{t}|\mathcal{F}_{t}^{W^{0}}])\Big]dt+Z^{0,\xi}_{t}dW_{t}^{0},\\
   \mathbb{E}[X_{0}^{\xi}|\mathcal{F}_{0}^{W^{0}}] & = \mathbb{E}[\xi],\quad \mathbb{E}[Y_{T}^{\xi}|\mathcal{F}_{T}^{W^{0}}] = h_{1}\mathbb{E}[X_{T}^{\xi}|\mathcal{F}_{T}^{W^{0}}]+h_{2}(\mathbb{E}[X_{T}^{\xi}|\mathcal{F}_{T}^{W^{0}}]).
\end{aligned}\label{conditional expectation FBSDE}\right.
\end{equation}
\begin{lemma}\label{map lemma}
Suppose assumptions $ (A1)$ and $ (A3) $ hold, then FBSDE \eqref{conditional expectation FBSDE} admits a unique solution $(\mathbb{E}[X^{\xi}|\mathcal{F}^{W^{0}}],\mathbb{E}[Y^{\xi}|\mathcal{F}^{W^{0}}],Z^{0,\xi}) \in S^{2}_{\mathbb{F}^{0}}([0,T];\mathbb{R})\times S^{2}_{\mathbb{F}^{0}}([0,T];\mathbb{R}^{n})\times L^{2}_{\mathbb{F}^{0}}([0,T];\mathbb{R}^{n \times d})$. Moreover, for any $\xi_{1},\xi_{2}\in L^{2}_{\mathcal {F}_{0}}$, \begin{equation}\label{expectation uniform}
 \left |\mathbb{E}[Y_{t}^{\xi_{1}}|\mathcal{F}_{t}^{W^{0}}]-\mathbb{E}[Y_{t}^{\xi_{2}}|\mathcal{F}_{t}^{W^{0}}]\right|\leq C\left|\mathbb{E}[X_{t}^{\xi_{1}}|\mathcal{F}_{t}^{W^{0}}]-\mathbb{E}[X_{t}^{\xi_{2}}|\mathcal{F}_{t}^{W^{0}}]\right|,\quad \forall t\in[0,T],
\end{equation}
where C is constant only depending on $n,K,T$.
\end{lemma}
\begin{proof}
It can be verified that under assumptions (A1) and (A3) the monotonicity conditions in \cite{hua2022unified} are satisfied, which ensures the FBSDE \eqref{conditional expectation FBSDE} has a unique solution and \eqref{expectation uniform} holds.
\end{proof}

\indent Now we are ready to give the proof of Theorem \ref{global general theorem}.

\begin{proof}[Proof of Theorem \ref{global general theorem}]
By Lemma \ref{map lemma}, the map $\mathcal{M}$ has  a unique fixed point $(\mu,\nu)$. As explained above, solving \eqref{general fixed fbsde} with this $(\mu,\nu)$ as input, and denoting by $(X^{\xi},Y^{\xi},Z^{\xi})$ the resulting solution, by definition of a fixed point, we have $\mathbb{E}[X_{t}^{\xi}|\mathcal{F}_{t}^{W^{0}}] = \mu_{t}$ and $\mathbb{E}[Y_{t}^{\xi}|\mathcal{F}_{t}^{W^{0}}]=\nu_{t}$. We conclude that $(X_{t}^{\xi},Y_{t}^{\xi},Z_{t}^{\xi})_{0\leq t \leq T}$ is the unique solution of \eqref{conditional fbsde}.
\end{proof}

We conclude this subsection by giving a representation result for conditional mean field FBSDE \eqref{conditional fbsde}. 
 For any $t\in[0,T],\eta \in L^{2}_{\mathcal{F}^{W^{0}}_{t}}$, we now introduce the following FBSDE
 \begin{equation}\label{phi FBSDE}
 \left\{
     \begin{aligned}
      d \nu_{s}^{t,\eta} &= \Big[(b_{1}(s) +b_{2}(s) P_{s}) \nu_{s}^{t,\eta} + b_{2}(s)\varphi_{s}^{t,\eta} + b_{0}(s,\nu_{s}^{t,\eta},P_{s}\nu_{s}^{t,\eta}+\varphi_{s}^{t,\eta})\Big]ds\\
      &\quad +\Big[\sigma_{1}\nu_{s}^{t,\eta}+(P_{s}\nu_{s}^{t,\eta}+\varphi_{s}^{t,\eta})^{\mathsf{T}}\sigma_{2}+\sigma_{0}(s,\nu_{s}^{t,\eta},P_{s}\nu_{s}^{t,\eta}+\varphi_{s}^{t,\eta})\Big]dW_{s}^{0},\\
      d\varphi_{s}^{t,\eta}&=-\Big[ f_{2}(s)\varphi_{s}^{t,\eta}+(b_{2}(s)\varphi_{s}^{t,\eta})P_{s} +f_{0}(s,\nu_{s}^{t,\eta},P_{s}\nu_{s}^{t,\eta}+\varphi_{s}^{t,\eta})\\
      &\quad+P_{s}b_{0}(s,\nu_{s}^{t,\eta},P_{s}\nu_{s}^{t,\eta}+\varphi_{s}^{t,\eta})\Big]ds+z^{t,\eta}_{s}dW_{s}^{0},\\
      \nu_{t}&=\eta,\quad \varphi_{T} = h_{2}(\nu_{T}^{t,\eta}).
     \end{aligned}
     \right.
 \end{equation}
 \begin{Theorem}\label{phi representation theorem}
 Let assumptions $(A1)-(A3)$ hold, then for any $(t,\eta) \in [0,T]\times L^{2}_{\mathcal{F}_{t}^{W^{0}}}$, FBSDE \eqref{phi FBSDE} admits a unique solution and it satisfies that for any $s\in[t,T]$,
 \begin{equation}
   \nu^{t,\eta}_s= \mathbb{E}[X_{s}^{t,\xi}|\mathcal {F}^{W^{0}}_{s}],\quad \varphi_{s}^{t,\eta}=Y_{s}^{t,\xi}-P_{s}X_{s}^{t,\xi},\label{u x v}
   \end{equation}
where $X^{t,\xi},Y^{t,\xi}$ are the first two components of the unique solution of conditional mean field FBSDE \eqref{conditional fbsde} with initial condition $\xi\in L^{2}_{\mathcal{F}_{t}}$ satisfying $\mathbb{E}[\xi|\mathcal{F}_{t}^{W^{0}}] = \eta$ and $P$ is the unique solution of \eqref{p equation}.\label{phi theorem}
\end{Theorem}
\begin{proof}
    The proof follows immediately by a  slight modification of the proof of Lemma \ref{fix representation}. Under assumptions (A1)-(A3), it follows from Theorem \eqref{global general theorem} that  $ (\mathbb{E}[X^{t,\xi}_{s}|\mathcal{F}^{W^{0}}_{s}],\mathbb{E}[Y_{s}^{t,\xi}|\mathcal{F}_{s}^{W^{0}}])$ is the fixed point of the map $\mathcal{M}$. Replacing the fixed processes $(\mu_{s},\nu_{s})$ in SDE \eqref{SDE} and BSDE \eqref{BSDE} by $(\mathbb{E}[X^{t,\xi}_{s}|\mathcal{F}^{W^{0}}_{s}],\mathbb{E}[Y_{s}^{t,\xi}|\mathcal{F}_{s}^{W^{0}}])$, we get
    \begin{equation}
 \left\{
 \begin{aligned}
  dX_{s}^{t,\xi}&=\left[ b_{1}X_{s}^{t,\xi}+b_{2}(P_{s}X_{s}^{t,\xi}+\varphi_{s}^{t,\xi})+b_{0}(s,\mathbb{E}[X^{t,\xi}_{s}|\mathcal{F}^{W^{0}}_{s}],\mathbb{E}[Y^{t,\xi}_{s}|\mathcal{F}^{W^{0}}_{s}])\right]ds+\sigma_{s}dW_{s}\\
      &\quad +\Big[\sigma_{1}(s)X_{s}^{t,\xi}+(P_s X_s^{t,\xi}+\varphi_s^{t,\xi})^{\mathsf{T}}\sigma_{2}(s)+\sigma_{0}(s,\mathbb{E}[X^{t,\xi}_{s}|\mathcal{F}^{W^{0}}_{s}],\mathbb{E}[Y^{t,\xi}_{s}|\mathcal{F}^{W^{0}}_{s}]) \Big]dW_{s}^{0}\\
    d\varphi_{s} ^{t,\xi}&=-\left[ f_{2}(s)\varphi_{s}^{t,\xi}+(b_{2}(s)\varphi_{s}^{t,\xi})P_{s} +f_{0}(s,\mathbb{E}[X^{t,\xi}_{s}|\mathcal{F}^{W^{0}}_{s}],\mathbb{E}[Y^{t,\xi}_{s}|\mathcal{F}^{W^{0}}_{s}])\right.\\
    &\quad\left. +P_{s}b_{0}(s,\mathbb{E}[X^{t,\xi}_{s}|\mathcal{F}^{W^{0}}_{s}],\mathbb{E}[Y^{t,\xi}_{s}|\mathcal{F}^{W^{0}}_{s}])\right] ds+\Gamma_{s}^{t,\xi}dW_{s}^{0},   \\
    X_{t}&=\xi, \quad \varphi_{T}^{t,\xi}=h_{2}(\mathbb{E}[X^{t,\xi}_{T}|\mathcal{F}^{W^{0}}_{T}]). 
 \end{aligned}
  \right.\label{xphi_FBSDE}
 \end{equation}
 By Lemma \ref{fix representation}, we obtain that 
    \begin{equation*}
     \varphi_s^{t,\xi}  = Y_s^{t,\xi}- P_s X_s^{t,\xi}. 
    \end{equation*}
Taking conditional expectation in above equation with respect to $\mathcal{F}^{W^{0}}_{s}$, we obtain
\begin{equation} \label{phi linear}
  \varphi_s^{t,\xi}  =\mathbb{E}[Y_s^{t,\xi}|\mathcal{F}^{W^{0}}_{s}]- P_s \mathbb{E}[Y_s^{t,\xi}|\mathcal{F}^{W^{0}}_{s}].  
\end{equation}
Moreover, taking conditional expectation in the forward equation of \eqref{xphi_FBSDE} with respect to $\mathcal{F}^{W^{0}}_{s}$, we get
\begin{equation} \label{nu sde}
\begin{aligned}
  d\mathbb{E}[X^{t,\xi}_{s}|\mathcal{F}^{W^{0}}_{s}]&=\left[ b_{1}\mathbb{E}[X^{t,\xi}_{s}|\mathcal{F}^{W^{0}}_{s}]+b_{2}(P_{s}\mathbb{E}[X^{t,\xi}_{s}|\mathcal{F}^{W^{0}}_{s}]+\varphi_{s}^{t,\xi})+b_{0}(s,\mathbb{E}[X^{t,\xi}_{s}|\mathcal{F}^{W^{0}}_{s}],\mathbb{E}[Y^{t,\xi}_{s}|\mathcal{F}^{W^{0}}_{s}])\right]ds\\
  &\quad +\Big[\sigma_{1}(s)\mathbb{E}[X^{t,\xi}_{s}|\mathcal{F}^{W^{0}}_{s}]+(P_s \mathbb{E}[X^{t,\xi}_{s}|\mathcal{F}^{W^{0}}_{s}]+\varphi_s^{t,\xi})^{\mathsf{T}}\sigma_{2}(s)\\
  &\quad +\sigma_{0}(s,\mathbb{E}[X^{t,\xi}_{s}|\mathcal{F}^{W^{0}}_{s}],\mathbb{E}[Y^{t,\xi}_{s}|\mathcal{F}^{W^{0}}_{s}]) \Big]dW_{s}^{0}. \end{aligned}
\end{equation}
Combining \eqref{xphi_FBSDE}, \eqref{nu sde} and \eqref{phi linear}, we get that $\mathbb{E}[X^{t,\xi}_{s}|\mathcal{F}^{W^{0}}_{s}], Y_{s}-P_{s}X^{t,\xi}_{s}$ are the first two components of a solution of FBSDE \eqref{phi FBSDE} and the uniqueness property is guaranteed by the classical FBSDEs results in small intervals and the result of Lemma \ref{p lemma} and Lemma \ref{map lemma}.
\end{proof}

\subsection{Classical solution of the associated master equation}\label{section 3.2}
In this subsection, we will consider classical solution for the master equation related to conditional mean field FBSDE \eqref{conditional fbsde}. More precisely, we study classical solution of the following equation
\begin{equation}
\left\{\begin{aligned}
&\partial_t U(t, x, \nu)+\partial_x U(t, x, \nu)\Big[b_{1}(t)x+b_{2}(t) U(t, x, \nu)+b_{0}\left(t,\nu,\mathbb{E}[U(t, \xi, \nu)]\right)\Big]\\
&+\frac{1}{2} \partial_{x x} U(t, x, \nu)\Big|\sigma_{1}x+U(t,x,\nu)^{\mathsf{T}}\sigma_{2}+\sigma_{0}(t,\nu,\mathbb{E}[U(t,\xi,\nu)])\Big|^{2}+\frac{1}{2}\partial_{x x} U(t, x, \nu)\sigma^{2}\\
&+\partial_\nu U(t, x, \nu)\Big[b_{1}(t)\nu+b_{2}(t) \mathbb{E}[U(t, \xi, \nu)]+b_{0}(t,\nu,\mathbb{E}[U(t, \xi, \nu)])\Big] \\
&+\frac{1}{2}\partial_{\nu\nu}U(t,x,\nu)\Big|\sigma_{1}\nu+\mathbb{E}[U(t,\xi,\nu)]^{\mathsf{T}}\sigma_{2}+\sigma_{0}(t,\nu,\mathbb{E}[U(t,\xi,\nu)])\Big|^{2}\\
&+\partial_{x\nu}U(t,x,\nu)\Big[\left(\sigma_{1}x+U(t,x,\nu)^{\mathsf{T}}\sigma_{2}+\sigma_{0}(t,\nu,\mathbb{E}[U(t,\xi,\nu)])\right)\\
&\quad\left(\sigma_{1}\nu+\mathbb{E}[U(t,\xi,\nu)]^{\mathsf{T}}\sigma_{2}+\sigma_{0}(t,\nu,\mathbb{E}[U(t,\xi,\nu)])\right)^{\mathsf{T}}\Big]\\
&+f_{1}(t)x+f_{2}(t)U(t,x,\nu)+f_{0}(t,\nu,\mathbb{E}[U(t,\xi,\nu)])=0, \\
&U(T, x, \nu)=h_{1}x+h_{2}(\nu),
\end{aligned}\right.\label{general master equation}
\end{equation}
where $\xi \in L^{2}_{\mathcal{F}_{0}}$ with $\mathbb{E}[\xi] = \nu$.  We further introduce the following assumption:\\
\textbf{Assumption (A4):} The functions $f_{0},b_{0},\sigma_{0},\sigma,h_{2}$ are deterministic satisfying for any $t\in[0,T]$,
\begin{equation*}
    |b_0(t,0,0)|\leq K,~|f_0(t,0,0)|\leq K,~|\sigma_{0}(t,0,0)|\leq K,~|\sigma(0)|\leq K,~|h_2(0)|\leq K,
\end{equation*}
 and $f_{0}(t,\cdot,\cdot)\in C^{2}(\mathbb{R}\times \mathbb{R}^{n};\mathbb{R}^n)$, $b_{0}(t,\cdot,\cdot)\in C^{2}(\mathbb{R}\times \mathbb{R}^{n};\mathbb{R})$, $ \sigma_{0}(t,\cdot,\cdot)\in C^{2}(\mathbb{R}\times \mathbb{R}^{n};\mathbb{R}^{d})$ and $h_{2}\in C^{2}(\mathbb{R};\mathbb{R}^{n})$.

From now we consider in Markovian setting and define the function $\Phi:[0,T]\times\mathbb{R}\rightarrow \mathbb{R}^{n}$ as
\begin{equation}
    \Phi(t,\nu) = \varphi_{t}^{t,\nu},\label{phi decoupling}
\end{equation}
where $\varphi^{t,\nu}_{t}$ is the second component of the unique solution of FBSDE \eqref{phi FBSDE} with initial condition $\nu\in \mathbb{R}$, and further it holds that $\varphi_{s}^{t,\nu}=\Phi(s,\nu_{s}^{t,\nu})$, for all $s \in [t,T]$ (see \cite{delarue2002existence}), which implies that $\Phi(t,\nu)$ 
corresponds to the following PDE
\begin{equation}
\left\{
    \begin{aligned}
     &\partial_{t}\Phi(t,\nu)+\partial _{\nu}\Phi(t,\nu)\Big[(b_{1}(t)+b_{2}(t)P_{t})\nu+b_{2}(t)\Phi(t,\nu)+b_{0}(t,\nu,P_{t}\nu+\Phi(t,\nu))\Big]\\
     &+\frac{1}{2}\partial_{\nu \nu}\Phi(t,\nu)\Big|\sigma_{1}(t)\nu+(P_{t}\nu+\Phi(t,\nu))^{\mathsf{T}}\sigma_{2}(t)+\sigma_{0}(t,\nu,P_{t}\nu+\Phi(t,\nu)\Big|^{2}\\
     &+(b_{2}(t) P_{t})\Phi(t,\nu)+f_{2}(t)\Phi(t,\nu)+f_{0}(t,\nu,P_{t}\nu+\Phi(t,\nu))+P_{t}b_{0}(t,\nu,P_{t}\nu+\Phi(t,\nu))=0,\\
     &\Phi(T,\nu) = h_{2}(\nu).
    \end{aligned}\right.\label{phi pde}
\end{equation}

We would like to show that 
 $\Phi\in C^{1,2}([0,T]\times\mathbb{R};\mathbb{R}^{n})$ to verify $\Phi$ is indeed a classical solution to \eqref{phi pde}. Now let us consider the following FBSDE on $[t,T]$, which can be interpreted as a formal differentiation of \eqref{phi FBSDE} with its initial condition and we choose $\nu\in \mathbb{R}$ as its initial condition.
\begin{equation}
\left\{\begin{aligned}
d \nabla \nu_s^{t, \nu}= & \Big[(b_{1}(s)+b_{2}(s) P_{s})\nabla\nu_{s}^{t,\nu}+b_{2}(s) \nabla\varphi_{s}^{t,\nu}\\
&+\partial_{\nu}b_{0}(s,\nu_{s}^{t,\nu},P_{s}\nu^{t,\nu}_{s}+\varphi_{s}^{t,\nu})\nabla\nu_{s}^{t,\nu}+\partial_{\varphi}b_{0}(s,\nu_{s}^{t,\nu},P_{s}\nu^{t,\nu}_{s}+\varphi_{s}^{t,\nu})\nabla\varphi_{s}^{t,\nu}\Big]ds\\
& + \Big[\sigma_{1}(s)\nabla \nu_{s}^{t,\nu}+(\nabla \varphi_{s}^{t,\nu})^{\mathsf{T}}\sigma_{2}(s)+\partial_{\nu}\sigma_{0}(s,\nu_{s}^{t,\nu},P_{s}\nu^{t,\nu}_{s}+\varphi_{s}^{t,\nu})\nabla\nu_{s}^{t,\nu}\\
&+\partial_{\varphi}\sigma_{0}(s,\nu_{s}^{t,\nu},P_{s}\nu^{t,\nu}_{s}+\varphi_{s}^{t,\nu})\nabla\varphi_{s}^{t,\nu}\Big]dW_{s}^{0}\\
d \nabla \varphi_s^{t, \nu}= & \Big[- 
f_{2}(s)\nabla \varphi_{s}^{t,\nu} -P_{s}( b_{2}(s)\nabla \varphi_{s}^{t,\nu})-\partial_{\nu}f_{0}(s,\nu_{s}^{t,\nu},P_{s}\nu_{s}^{t,\nu}+\varphi_{s}^{t,\nu})\nabla\nu_{s}^{t,\nu}\\
&-\partial_{\varphi}f_{0}(s,\nu_{s}^{t,\nu},P_{s}\nu_{s}^{t,\nu}+\varphi_{s}^{t,\nu}) \nabla\varphi_{s}^{t,\nu}-P_{s}\partial_{\nu}b_{0}(s,\nu_{t}^{t,\nu},P_{s}\nu_{s}^{t,\nu}+\varphi_{s}^{t,\nu})\nabla\nu_{s}^{t,\nu}\\
&-P_{s}\partial_{\varphi}b_{0}(s,\nu_{s}^{t,\nu},P_{s}\nu_{s}^{t,\nu}+\varphi_{s}^{t,\nu})\nabla\varphi_{s}^{t,\nu}\Big]ds+\nabla z^{t,\nu}_{s}dW_{s}^{0},\\
\nabla \nu_{t}^{t,\nu} =& 1, \nabla \varphi_T^{t, \nu}=h_{2}^{\prime}\left(\nu_T^{t, \nu}\right) \nabla \nu_T^{t, \nu}.
\end{aligned}\right.\label{formal differentiation}
\end{equation}
\begin{Theorem}
Under assumptions $(A1)-(A4)$, the function $\Phi \in C^{1,2}([0, T] \times \mathbb{R};\mathbb{R}^{n})$ defined as \eqref{phi decoupling} is the unique classical solution to \eqref{phi pde} with bounded $\partial_\nu \Phi, \partial_{\nu\nu}\Phi$.
\end{Theorem}
\begin{proof} From Theorem \ref{phi representation theorem}, we know
\begin{equation}
     \Phi\left(t, \nu\right)=Y_{t}^{t,\xi}-P_{t}X_{t}^{t,\xi}, \quad \forall t\in[0,T],\label{phi definition}
 \end{equation}
 where $X^{t,\xi},Y^{t,\xi}$ are the first two components of the unique solution of conditional mean field FBSDE \eqref{conditional fbsde} with initial condition $\xi\in L^{2}_{F_{0}}$ satisfying $\mathbb{E}[\xi] = \nu$ and $P$ is the unique solution of \eqref{p equation}.
 Taking conditional expectation on both sides of
\eqref{phi definition}, we get
\begin{equation*}
\Phi(t,\nu)=
    \mathbb{E}[Y_{t}^{t,\xi}|\mathcal{F}^{W^{0}}_{t}] - P_{t}\nu, \quad \forall t\in[0,T].
\end{equation*}
Then from Lemma \ref{p lemma} and Lemma \ref{map lemma},  we obtain that  
\begin{equation}
    \left|\Phi(t,\nu_{1})-\Phi(t,\nu_{2})\right| \leq C\left|\nu_{1}-\nu_{2}\right|, \forall \nu_{1},\nu_{2} \in \mathbb{R},\quad \forall t\in[0,T],\label{phi uniform}
\end{equation}
where C only depending on $n,K,T$. 

Now we would like to show that $\Phi\in C^{1,2}([0,T]\times \mathbb{R};\mathbb{R}^{n})$ with bounded $\partial_{\nu}\Phi$, $\partial_{\nu\nu}\Phi$. First, we show $\Phi \in C^{1}([0,T]\times \mathbb{R};\mathbb{R}^{n})$ with bounded $\partial_{\nu}\Phi$.
 Let us consider the linear FBSDE \eqref{formal differentiation} on $\left[t, T\right]$ for any $t\in[0, T)$. Note that under assumptions (A1)-(A3), all the coefficients in FBSDE \eqref{formal differentiation} are bounded by some chosen $C_1 \geq C$. By standard FBSDE arguments, there exists some $\delta_0>0$ depending on $C_{1}$ such that the FBSDE \eqref{formal differentiation} is well-posed on $\left[T-\delta_0, T\right]$, which implies that $\partial_{\nu} \Phi \in C^{0}([T-\delta,T]\times \mathbb{R};\mathbb{R}^{n})$. Combined with \eqref{phi uniform}, following standard arguments, we obtain that $\Phi\in C^{1}([T-\delta,T]\times \mathbb{R};\mathbb{R}^{n})$. We then consider the FBSDE \eqref{formal differentiation} with $h_{2}(\cdot)$ replaced by $\Phi\left(T-\delta_0, \cdot\right)$. According to \eqref{phi uniform}, the FBSDE \eqref{formal differentiation} is also well-posed on $\left[T-2 \delta_0, T-\delta_0\right]$. Repeating this procedure backwardly and finitely many
times, we are able to show that the FBSDE \eqref{formal differentiation} is well-posed on $\left[t, T\right]$ for any $t\in[0,T]$ and  $\Phi \in C^{1}([0, T] \times \mathbb{R};\mathbb{R}^{n})$ with bounded $\partial_{\nu}\Phi$. Next, by differentiating \eqref{formal differentiation} in $\nu$ again and using \eqref{phi uniform}, we can further show that $\Phi \in C^{1,2}([0, T] \times\mathbb{R};\mathbb{R}^{n})$ with bounded $\partial_{\nu\nu}\Phi$. Consequently, $\Phi$ is a classical solution of PDE \eqref{phi pde}.\\
\textbf{Uniqueness:} Suppose that $\tilde{\Phi}\in C^{1,2}([0, T] \times \mathbb{R};\mathbb{R}^{n})$ is another classical solution to \eqref{phi pde} with bounded $\partial_\nu \tilde{\Phi},\partial_{\nu\nu} \tilde{\Phi}$. For any $\left(t, \nu\right) \in[0, T] \times \mathbb{R}$, we first consider the following well-posed SDE
\begin{equation*}
\left\{\begin{aligned}
d \tilde{\nu}_s^{t, \nu} &= \left[ (b_{1}(s) +b_{2}(s) P_{s})\tilde{\nu}_s^{t, \nu} + b_{2}(s) \tilde{\Phi}(s,\tilde{\nu}_s^{t, \nu}) +b_{0}(s,\tilde{\nu}_s^{t, \nu},P_{s}\tilde{\nu}_s^{t, \nu}+\tilde{\Phi}(s,\tilde{\nu}_s^{t, \nu}))\right]ds\\
&\quad+\Big[\sigma_{1}(s)\tilde{\nu}_s^{t, \nu}+(P_{s}\tilde{\nu}_s^{t, \nu}+\tilde{\Phi}(s,\tilde{\nu}_s^{t, \nu}))^{\mathsf{T}}\sigma_{2}(s)+\sigma_{0}(s,\tilde{\nu}_s^{t, \nu},P_{s}\tilde{\nu}_s^{t, \nu}+\tilde{\Phi}(s,\tilde{\nu}_s^{t, \nu}))\Big]dW_{s}^{0}\\
\tilde{\nu}_{t}^{t, \nu} &=   \nu.
\end{aligned}\right.
\end{equation*}
Let $\tilde{\varphi}_s^{t, \nu}:=\tilde{\Phi}\left(s, \tilde{\nu}_s^{t, \nu}\right)$. Since $\tilde{\Phi}$ is a classical solution to \eqref{phi pde}, it can be easily checked that $\tilde{\varphi}^{t, \nu}$ solves the BSDE in \eqref{phi FBSDE}. Therefore we have verified that $\left(\tilde{\nu}^{t, \nu}, \tilde{\varphi}^{t, \nu}\right)$ are the first two components of a solution to  FBSDE \eqref{phi FBSDE} with initial condition $\nu$. Therefore, the uniqueness result follows by the well-posedness of the FBSDE \eqref{phi FBSDE}.
\end{proof}

\begin{Theorem}\label{master theorem}
 Let Assumptions $(A1)-(A4)$ hold, then the function 
 \begin{equation}
U(t,x,\nu):=P_{t}x+\Phi(t,\nu)\label{U definition}
 \end{equation}
is the unique classical solution to the master equation \eqref{general master equation} with bounded $\partial_{x}U,\partial_{\nu}U, \partial_{\nu\nu}U$.
\end{Theorem}
\begin{proof}
\textbf{Existence:} First, by Theorem \ref{phi representation theorem}, we know for $t\in[0,T]$, $s\in[t,T]$, 
\begin{equation}
    Y_{s}^{t,\xi} = P_{s}X^{t,\xi}_{s} + \varphi_{s}^{t,\nu}= P_{s}X^{t,\xi}_{s}+\Phi(s,\mathbb{E}[X^{t,\xi}_{s}|\mathcal{F}_{s}^{0}]) = U(t,X^{t,\xi}_{s},\mathbb{E}[X_{s}^{t,\xi}|\mathcal{F}^{W^{0}}_{s}]),\label{U decoupling}
\end{equation}
where $\xi \in L^{2}_{\mathcal{F}_{0}}$ with $\mathbb{E}[\xi] = \nu$ and $(X^{t,\xi},Y^{t,\xi})$ is the first two components of solution of conditional mean field FBSDE \eqref{conditional fbsde} with initial condition $\xi$ and $\varphi^{t,\nu}$ is the second component of solution for FBSDE \eqref{phi FBSDE} with initial condition $\nu\in \mathbb{R}$. Therefore, the function $U(t,x,\nu)$ is the decoupling field of conditional mean field FBSDE \eqref{conditional fbsde}. \\
\indent Next, we verify the decoupling field $U$ satisfies the master equation \eqref{general master equation}. We first check that $U$ satisfies the terminal condition
\begin{equation*}
U(T, x, \nu)=P_T x+\Phi(T, \nu)=h_{1}x+h_{2}(\nu).
\end{equation*}
Moreover, it follows from \eqref{U decoupling} by setting $s=t$ and taking conditional expectation with respect to $\mathcal{F}^{W^{0}}_{t}$ that
\begin{equation}
P_{t}\nu+\Phi(t,\nu)=\mathbb{E}[U(t,\xi,\nu)],\label{u expectation}
\end{equation}
where $\xi \in L^{2}_{\mathcal{F}_{0}}$ with $\mathbb{E}[\xi] = \nu$.\\
Recalling \eqref{p equation} and \eqref{phi pde}, we obtain
\begin{equation}
\begin{aligned}
 \partial_t U(t, x, \nu)&=\partial_t P_t x+\partial_t \Phi(t, \nu) \\
& =-\Big[(b_{2}(t) P_{t})P_{t}+f_{2}(t) P_{t}+b_{1}(t)P_{t}+f_{1}(t)\Big] x-\partial_{\nu}\Phi(t,\nu)\Big[(b_{1}(t)+b_{2}(t)P_{t})\nu\\
&\quad \quad+b_{2}(t)\Phi(t,\nu)+b_{0}(t,\nu,P_{t}\nu+\Phi(t,\nu))\Big]\\
&\quad-\frac{1}{2}\partial_{\nu \nu}\Phi(t,\nu)\Big|\sigma_{1}(t)\nu+(P_{t}\nu+\Phi(t,\nu))^{\mathsf{T}}\sigma_{2}(t)+\sigma_{0}(t,\nu,P_{t}\nu+\Phi(t,\nu))\Big|^{2}\\
&\quad-\Big[(b_{2}(t))\Phi(t,\nu))P_{t}+f_{2}(t)\Phi(t,\nu)+f_{0}(t,\nu,P_{t}\nu+\Phi(t,\nu))+P_{t}b_{0}(t,\nu,P_{t}\nu+\Phi(t,\nu))\Big].
\end{aligned}\label{master equation2}
\end{equation}
Moreover, we have
\begin{equation}
\begin{aligned}
\partial_x U(t, x, \nu)&=P_t, \quad \partial_{x x} U(t, x, \nu)=\partial_{x\nu}U(t,x,\nu) = 0,\\
\quad \partial_\nu U(t, x, \nu)&=\partial_\nu \Phi(t, \nu), \quad \partial_{\nu\nu}U(t,x,\nu) = \partial_{\nu\nu} \Phi(t,\nu)   
\end{aligned}
\end{equation}
are all bounded. Plugging the above terms into \eqref{general master equation} and using \eqref{U definition}, \eqref{u expectation}, it is straightforward to show that $U$ is a classical solution to the master equation \eqref{general master equation}.\\
\textbf{Uniqueness:} We recall that the solution $U$ to \eqref{general master equation} serves as the decoupling field of conditional mean field FBSDE \eqref{conditional fbsde}. Following the uniqueness argument in Theorem \ref{global general theorem}, the well-posedness of \eqref{conditional fbsde} implies the uniqueness of a solution to the master equation \eqref{general master equation}.
\end{proof}
\begin{remark}
It is worth emphasizing that we establish global well-posedness of master equation with nonseparable Hamiltonian and our monotonicity condition is dichotomy with the well-known Lasry-Lions motononicity and displacement monotonocity conditions. Compared with our previous work \cite{hua2023well}, the master equation \eqref{general master equation} involves the second-order derivative term of $U$ with respect to $\nu$ since the presence of the common noises, which further requires us to prove second order differentiability of function $\Phi$.
\end{remark}
\bibliographystyle{siam}
\bibliography{bib}
\end{document}